\documentclass{article}
\usepackage{amssymb}
\usepackage{amsthm}
\usepackage{amsmath}
\usepackage{enumerate}
\usepackage{setspace}
\usepackage{bm}
\usepackage{mathrsfs}
\usepackage{graphicx}
\usepackage{color}
\usepackage[hypertex]{hyperref}
\usepackage[all]{xy}
\theoremstyle{definition}
\newtheorem{thm}{Theorem}
\newtheorem*{thm*}{Theorem}
\newtheorem{lem}[thm]{Lemma}
\newtheorem{prop}[thm]{Proposition}
\newtheorem{dfn}[thm]{Definition}
\newtheorem{cor}[thm]{Corollary}
\newtheorem{exm}[thm]{Example}
\newtheorem{rmk}[thm]{Remark}

\newcommand{\T}{\mathbb{T}}
\newcommand{\R}{\mathbb{R}}

\newcommand{\Z}{\mathbb{Z}}
\newcommand{\p}{\mathbb{P}}

\newcommand{\Cok}{\operatorname{Cok}}

\begin{document}

\title{Tropical Theta Functions and Riemann-Roch 
Inequality for Tropical Abelian Surfaces}


\author{Ken Sumi}


\maketitle

\begin{abstract}
We show that the space of 
theta functions on tropical tori is identified with a convex polyhedron.
We also show a Riemann-Roch inequality for tropical abelian surfaces
by calculating the self-intersection numbers of
divisors.
\end{abstract}

\section{Introduction}


 %
Tropical geometry is a field of mathematics studying piecewise-linear objects
that appear as certain degenerate limits of algebraic varieties.
Several results in algebraic geometry have analogies
in tropical geometry.
One of the famous analogies is a Riemann-Roch theorem for compact tropical curves.

A Riemann-Roch theorem for finite graphs was 
discovered by Baker-Norine \cite{BN} in 2007 and extended to tropical curves
by Gathmann-Kerber \cite{GK} and Mikhalkin-Zharkov \cite{MZ} in 2008.
The tropical Riemann-Roch theorem states that,
for a compact tropical curve $C$ of genus $g$ and a tropical divisor $D$ on $C$,
we have
\[ r(D)-r(K-D)=\deg D -g+1 . \]
The number $r(D)$  is a substitute for the dimension of 
the complete linear system $|D|=\p (H^{0}(C,\mathcal{O}(D)))$,
called the rank of the divisor $D$.
In general, however,
$r(D)$ is different from the dimension of $|D|$ as a polyhedral complex. 

It is a very interesting problem to generalize the tropical Riemann-Roch 
theorem to higher dimensions. 
A main obstacle to higher dimensional generalization 
is to define the Euler characteristic 
$\chi (X,\mathcal{O}(D))$: 
firstly the higher cohomology $H^i(X,\mathcal{O}(D))$ with $i\ge 1$ 
is not defined in tropical geometry; 
secondly the dimension of the space $\Gamma (X,\mathcal{O}(D))$ 
of global sections of $\mathcal{O}(D)$ can be 
different from the number $r(D)+1$ appearing in the 
Riemann-Roch theorem.

Tropical versions of Noether's formula were shown by
Cartwright \cite{DC} for weak tropical surfaces
 and by Shaw \cite{S} for compact tropical surface in the sense of Mikhalkin.
They can be viewed as a Riemann-Roch theorem for the trivial divisor on a tropical surface. 
They avoid the above problems by considering the topological Euler characteristic of tropical surfaces.

Cartwright conjectured a Riemann-Roch inequality for weak tropical surfaces in \cite[Conjecture 3.6]{DC2}.
He introduced a higher dimensional analogue $h^{0}(X,D)$ of $r(D)+1$
and proposed a Riemann-Roch inequality by 
omitting $h^{1}(X, D)$ and assuming the Serre duality i.e. replacing 
$h^{2}(X, D)$ with $h^{0}(X, K-D)$,
where $K$ is the canonical divisor.

In this paper, we show the Cartwright conjecture for tropical abelian varieties, that is,
\begin{thm}[Corollary \ref{cor}]\label{THM2}
Let $X$ be a tropical abelian surface and let $D$ be a divisor on $X$.
Then the following inequality holds;
\[ h^{0}(X,D)+h^{0}(X,-D)\geq \frac{1}{2}D^{2}.\] 
\end{thm}

A relation between $h^{0}(X,D)$ and the topology of the set of regular section
$\Gamma (X,\mathcal{O}(D))$
is also an interesting question.
We obtained the following results.

\begin{thm}[Theorem \ref{THM}]\label{THM1}
Let $X=\R^n/\Lambda$ be a tropical torus and 
let $L \to X$ be a tropical line bundle. 
If $\Gamma (X,L) \neq \{-\infty\}$, 
$\Gamma (X,L)$ is identified with a convex polyhedron and
$\p(\Gamma (X,L))$ is identified with a compact convex 
polyhedron; in particular it is pure-dimensional. 
\end{thm}

\begin{thm}[Theorem \ref{rD}]\label{introthm1}
Let $X$ be a tropical torus and let $D$ be a divisor on $X$.
Then $h^{0}(X, D)$ defined by Cartwright \cite{DC2} coincides with
the dimension of $\Gamma (X,\mathcal{O}(D))$
as a convex polyhedron.
\end{thm}

 If a tropical analogue of Hirzebruch-Riemann-Roch theorem holds for
 $n$ dimensional tropical abelian varieties,
 then $\frac{1}{n!}D^{n}$ should be integral. 
 In the last section, we show this.
 \begin{prop}[Theorem \ref{morethm}]
 For a divisor $D$ on a tropical abelian variety, $\frac{1}{n!}D^{n}$ is an integer.
 \end{prop}

After the first version of the paper was put on the arXiv,
Dustin Cartwright kindly suggested to the author that
statements as in Theorem \ref{introthm1} and Theorem \ref{morethm} may hold.
We would like to thank him for his valuable suggestions.

\section{Tropical modules}
The tropical semifield $\T$ is the set $\R\cup \{ -\infty \}$ equipped with
the tropical sum $``x+y"=\max \{ x,y \}$ and
the tropical product $``xy"=x+y$ for $x,y\in \T$.
The set of invertible elements of $\T$ is $\R$.
A commutative semigroup $V$  
equipped with the unit $\{ -\infty\}$ and a scalar product $\T\times V \rightarrow V$ 
is called a {\it tropical module} or a $\T${\it -module}
if the following conditions are satisfied:
\begin{itemize}
\item $``(x+y)v"=``xv+yv"$ for any $x,y\in \T$, any $v\in V$;
\item $``x(v+w)"=``xv+xw"$ for any $x \in \mathbb{T}$, any $v,w\in V$;
\item $``x(yv)"=``(xy)v"$ for any $x,y\in \mathbb{T}$, any $v\in V$;
\item $``0\cdot v"=v$ for the multiplicative unit $0\in \T$ and any $v\in V$;
\item if $``xv"=``yv"$ for some $x,y \in \T$ and $v\in \T$,
then $x=y$ or $v=-\infty$.
\end{itemize}

By these conditions, we have $``-\infty\cdot v"=-\infty$.
For example, $\T ^{n}$ is naturally a $\T$-module.
Let $e_{i}\in \T^{n}$ be the element 
with 0 in the $i$th coordinate and $-\infty$ in the other coordinates. 

Let $V$, $W$ be tropical modules.
A map $f\colon V\to W$ is called a {\it tropical linear morphism} if
$f$ satisfies $f(``v_{1}+v_{2}")=``f(v_{1})+f(v_{2})"$ and $f(``tv")=``tf(v)"$ for any $v,v_{1},v_{2}\in V$, $t\in \T$.
If there is the inverse $f^{-1}$ and this is also a tropical linear morphism,
then we call $f$ an isomorphism.

Let $V$ be a tropical module.
Its projectivization $\p (V)$ is the quotient of $V\setminus \{ -\infty\}$
by the equivalence relation $\sim$
where $v\sim v' $ for $v,v'\in V$ if there exists $t\in \R$ such that $v'=``tv"$.

\begin{exm}\label{Vn}
Let $V_{n}$ be the tropical submodule of $\T^{n}$ generated by 
\[ f_{1}:=(-\infty, 0,\ldots ,0),f_{2}:=(0, -\infty, 0,\ldots ,0),
\ldots ,f_{n}:=(0,\ldots ,0,-\infty ).\]
For any linear combination
$a_{1}f_{1}+a_{2}f_{2}+\cdots +a_{n}f_{n}$ with
$a_{i}\geq a_{j}\geq \text{(the others)}$,
$a_{1}f_{1}+a_{2}f_{2}+\cdots +a_{n}f_{n}$ is the element
whose $i$th coordinate is  $a_{j}$ and the other coordinates are $a_{i}$.
For example, if $a_{1}\geq a_{2}\geq a_{k}$, $\ k=3,\ldots ,n$, 
we get $a_{1}f_{1}+a_{2}f_{2}+\cdots +a_{n}f_{n}=(a_{2},a_{1},\ldots ,a_{1})$.

The projectivization $\p (V_{n})\subset \p (\T^{n})$ is 
the finite graph with $n+1$ vertices
$[0,\ldots ,0],[ -\infty ,0,\ldots ,0] ,\ldots ,[0,\ldots ,0,-\infty ]$
such that the valency of the central vertex $[0,\ldots ,0]$ is $n$ and
the valency of the other vertex is $1$.
\end{exm}

\begin{exm}
A function $f:\R ^{n}\rightarrow \T$ is a {\it tropical polynomial} (resp. {\it tropical Laurent polynomial})
if $f$ is a constant map to $\{ -\infty\}$ or is
of the form 
$f(x)=\displaystyle\max_{j\in S}\{ a_{j}+j\cdot x \}$,
where $S$ is a finite subset of $(\Z_{\geq 0})^{n}$ (resp. $\Z^{n}$) and 
$a_{j}\in \R$. 
A tropical polynomial naturally extends to a function from $\T^{n}$ to $\T$.
The set of tropical polynomials on $\R^{n}$ is naturally a tropical module.
\end{exm}
Clearly, a tropical Laurent monomial is a {\it $\Z$-affine linear function}, that is, 
an affine linear function on $\R^{n}$ whose slope is in $(\Z ^{n})^{*}$. 

\section{Tropical Tori}
The goal of this section is to show the Theorem \ref{THM1}.

\subsection{Tropical manifold and line bundles}
In this section we briefly introduce some basic notions of tropical geometry
and give some examples.
References are made to Allermann-Rau \cite{AR}, Mikhalkin \cite{M} and
Mikhalkin-Zharkov \cite{MZ}, \cite{MZ2}.

\begin{dfn}
Let $U, V$ be open subsets of $\R^{n}, \R^{m}$ respectively.
A map $f:U\to V$ is {\it $\Z$-affine linear} 
if $f$ is given by a restriction of an affine map whose linear part is represented by an integral matrix.
\end{dfn}

\begin{dfn}\cite[Definition 1.22]{MG}
A {\it tropical } (or {\it integral affine}) {\it manifold} $M$ is a topological manifold
with an atlas $\{ (U_{i},\psi_{i}) \}$
such that $\psi_{i}\circ \psi_{j}^{-1}$ is a restriction of
a $\Z$-affine linear map on $\R^{n}$ for each $U_{i}\cap U_{j}\neq \emptyset$,
where $n=\operatorname{dim} M$.

A map $f: M_{1}\to M_{2}$ between tropical manifolds is {\it tropical}
if it is locally written as a $\Z$-affine map.
\end{dfn}

\begin{rmk}
In the literature, tropical manifolds usually mean more general spaces,
see Mikhalkin-Zharkov \cite[Definition 1.14]{MZ2}.
We adopt the above definition (following \cite[Definition 1.22]{MG})
 since we are only interested in tropical tori in this paper.
\end{rmk}

\begin{exm}
The Euclidean space $\R^{n}$ is a tropical manifold.
\end{exm}

\begin{exm}  
Fix the standard lattice $\Z^{n}\subset \R^{n}$.
Let $\Lambda\subset \R^{n}$ be a lattice, that is,
a discrete additive abelian subgroup of $\R^{n}$ of rank $n$.
The quotient $\R ^{n}/\Lambda $ is called a {\it tropical torus}.
This is a tropical manifold whose integral structure is given by the standard lattice $\Z^{n}$.
\end{exm}

\begin{dfn} 
Let $U$ be an open subset of a tropical manifold $M$.
A continuous function $f:U\to \T$ is said to be {\it regular} if 
$f$ is locally the restriction of a tropical Laurent polynomial on $\R^{n}$.

A continuous function $h:U\to \T$ is said to be {\it rational} if 
for all $x\in U$ there exists an open neighborhood $V\subset U$ of $x$ and
 two regular functions $f,\ g$ on $V$
such that $h|_{V}=``\frac{f}{g}"=f-g$.

The {\it structure sheaf} $\mathcal{O}_{U}$ of $U$ is the sheaf of regular functions on $U$.
The sheaf $\mathcal{O}_{U}^{*}$ is the subsheaf of $\mathcal{O}_{U}$ consisting of
invertible regular functions, that is, $\Z$-affine linear functions.
We denote $\mathcal{O}_{}^{*}$ by ${\rm Aff}_{\Z}(X)$.
This sheaf is locally constant.
\end{dfn}

\begin{rmk}  
Since the structure sheaf $\mathcal{O}_{X}$ is NOT an abelian sheaf,
we cannot consider the $i$th \v{C}ech cohomology group
$H^{i}(X,\mathcal{O}_{X})$ for $i>0$.
But since $\mathcal{O}_{X}^{*}$ is a multiplicative abelian sheaf,
we can consider $H^{1}(X,\mathcal{O}_{X}^{*})$ as in the classical case.
\end{rmk}

\begin{dfn}\label{linebundle}\cite[Definitions 1.5]{LA}  
Let $M $ be an $n$-dimensional tropical manifold.
A {\it tropical line bundle} on $M$ is a tuple $(L, \pi ,\{ U_{i},\Psi_{i} \}_{i})$ of a topological space
$L$, a continuous surjection $\pi : L \twoheadrightarrow M$,
an open covering $\{ U_{i} \}$ called the trivializing covering
and homeomorphisms $\Psi_{i}:\pi^{-1}(U_{i})\cong U_{i}\times \T$ 
called trivializations which satisfy:
\begin{itemize}
\item  
The following diagram is commute:
\[\xymatrix{ \pi^{-1}(U_{i}) \ar[r]^{\Psi_{i}} \ar[rd]_{\pi} & U_{i}\times \T 
\ar[d] 
 \\ & U_{i} } .\]
Here the vertical map is the first projection;
\item For every $i,j$ with $U_{i}\cap U_{j}\neq \emptyset$,
there exists a $\Z$-affine linear function  $\varphi_{ij} \colon U_{i}\cap U_{j}\to \R$ 
such that
$\Psi_{j}\circ\Psi_{i}^{-1}\colon (U_{i}\cap U_{j})\times \T\to (U_{i}\cap U_{j})\times \T$
is given by $(x,t)\mapsto (x,``\varphi_{ij}(x) t")$.
These $\varphi_{ij}$ are called a {\it transition functions}  of $L$.
\end{itemize}
It is clear that 
transition functions $\varphi _{ij}$ satisfies the cocycle condition $\varphi_{ij}\circ \varphi_{jk}=\varphi_{ik}$.
We identify two tropical line bundles $(L, \pi ,\{ U_{i},\Psi_{i} \}_{i\in I})$ and $(L, \pi ,\{ U_{j},\Psi_{j} \}_{j\in J})$
if $(L,\pi ,\{ U_{i},\Psi_{i} \}_{i\in I\sqcup J})$ is a tropical line bundle on $M$, that is, 
satisfies the second condition above.
We often write $L$ for $(L, \pi ,\{ U_{i},\Psi_{i} \}_{i})$ to avoid heavy notation.
\end{dfn}

We can refine the trivializing covering for $L$ as follows. 
If a covering $\{ V_{j} \}$ is a refinement  of the trivializing covering $\{ U_{i}\}$ for $L$,
then the trivializations $\Psi_{i}$ over $U_{i}$ induce trivializations $\Phi_{j}$ over $V_{j}$
and the line bundles $(L, \pi , \{ U_{i}, \Psi_{i} \})$ and $(L, \pi , \{ V_{i}, \Phi_{i} \})$ are identified.

Given an open covering $\{ U_{i} \}$ of $M$ and transition functions $\varphi_{ij} \colon U_{i}\cap U_{j}\to \R$
satisfying the cocycle condition,
we can construct a line bundle $L$ by
$L=\coprod_{i}U_{i}\times \T / \sim$.
Here 
$(x_{i},t_{i})\sim (x_{j},t_{j})$ for
$(x_{i},t_{i})\in U_{i}\times \T$ and $(x_{j},t_{j})\in U_{j}\times \T$
means $x_{i}=x_{j}\in U_{i}\cap U_{j}$ and $t_{j}=``\varphi_{ij}(x) t_{i}"$.

\begin{dfn}\cite[Lemma 1.16]{LA} 
Let $(L_{1},\pi _{1}, \{ U_{1i},\Psi_{1i} \})$ , $(L_{2},\pi _{2}, \{ U_{2j},\Psi_{2j} \})$ be two tropical line bundles
on a tropical manifold $M$.
By taking a common refinement of the trivializing coverings $\{ U_{1i}\} ,\{ U_{2i} \}$ if necessary,
we may assume that $U_{1i}=U_{2i}=:U_{i}$.
Then $L_{1}$ and $L_{2}$ are said to be {\it isomorphic}, denoted by $L_{1}\cong L_{2}$,
if  there exist a map $f\colon L_{1}\to L_{2}$ and invertible regular functions
$h_{ij}\colon U_{i}\cap U_{j}\to \T$ such that
$\Psi_{2j}\circ f\circ\Psi_{1i}^{-1}\colon (U_{i}\cap U_{j})\times \T\to (U_{i}\cap U_{j})\times \T$
is of the form $(x,t)\mapsto ( x, ``h_{ij}(x)t")$,

We can see that $L_{1}$ and $L_{2}$ are isomorphic
if and only if 
there exist 
an invertible regular function $g_{i}\colon U_{i}\to \R$ for each $i$
such that $``g_{i}(x)\varphi^{1}_{ij}(x)g_{j}(x)^{-1}"=\varphi^{2}_{ij}(x)$
for each $U_{i}\cap U_{j}\neq \emptyset$.
Here $\varphi^{1}_{ij}, \varphi^{2}_{ij}$ are the transition function of $L_{1}, L_{2}$ respectively.
\end{dfn}

\begin{dfn} 
We take two line bundles $L_{1}$ and $L_{2}$ as above.
The tensor product $L_{1}\otimes L_{2}$ is defined to be
the tropical line bundle on $M$ whose transition functions are 
$``\varphi^{1}_{ij}\varphi^{2}_{ij}"=\varphi^{1}_{ij}+\varphi^{2}_{ij}$.

The inverse $L_{1}^{-1}$ of the line bundle $L_{1}$ is defined to be the tropical line bundle on $M$
whose transition functions are 
$``(\varphi^{1}_{ij})^{-1}" =-\varphi^{1}_{ij}$.
\end{dfn}

\begin{dfn} 
Let $\pi\colon L\to M$ be a line bundle on a tropical manifold $M$ and 
$U$ be an open subset of $M$.
A function $s\colon U\to \pi^{-1}(U)$ is a {\it regular section} (resp. {\it rational section}) on $U$ of $L$
if $\pi\circ s$ is the identity map on $U$ and
$p_{i}\circ \Psi_{i}\circ s$ is a regular (resp. rational) function,
where $p_{i}\colon U_{i}\times \T \to \T$ is the second projection.
If $U=M$ then we call $s$ a global regular (resp. rational) section 
or simply a regular (resp. rational) section.
The set $\Gamma (U,L)$ of regular sections on $U$ of $L$ 
naturally has the structure of a tropical module.
\end{dfn}


The tropical Picard group ${\rm Pic}(M)$ is the set of 
isomorphic classes of tropical line bundles on $M$.
The product and inverse are given by the  tensor product and the inverse of line bundles, respectively.
It is clear that ${\rm Pic}(M)$ is naturally identified with 
$H^{1}(M,\mathcal{O}_{M}^{*})=H^{1}(M, {\rm Aff}_{\Z})$ via transition functions.


We can also define the pull-back of line bundles as in the classical case.
Let $f:M\to M'$ be a tropical map between tropical manifolds
and let $\pi: L\to M'$ be a tropical line bundle with the trivializing covering $\{  U_{i}\}$ 
and the transition functions $\varphi_{ij}$.
Then we define the pull-back as
\[ f^{*}L:=\{ (x ,\xi )\in M\times L | f(x)=\pi (\xi)  \}\] 
with the trivializing covering is $\{ f^{-1}(U_{i}) \}$ and the transition functions are $\varphi_{ij}\circ f$.

\subsection{Divisors on tropical tori}
In this section, we briefly introduce the notion of divisors.
For the details of divisors, we refer the reader to Allermann-Rau \cite{AR} or Mikhalkin \cite{M}.

\begin{dfn}
A {\it  rational polyhedron in $\R^{n}$} is a finite intersection of 
closed half spaces with rational slopes, that is, a subset of the form
$\bigcap_{i\in I}\{  x\in \R^{n} | a_{i}\cdot x\geq b_{i} \}$ 
where $I$ is a finite set and  $a_{i}\in \Z^{n}, b_{i}\in \R$.
A {\it face of a rational polyhedron $\sigma$} is a subset of the form 
$\sigma \cap \{ x\in \R^{n} | a\cdot x=b \}$
with $a\in \Z^{n}$ and $b\in \R$  satisfying $\sigma\subset \{ x\in \R^{n} | a\cdot x\geq b \}$

Let $M$ be a tropical manifold with an atlas $\{ (U_{i},\psi_{i}) \}$.
A {\it rational polyhedron in a tropical manifold $M$} is
a subset $\sigma\subset M$ such that
$\psi_{i} (\sigma \cap U_{i})$ is an intersection of $\operatorname{Im}\psi_{i}$
and a rational polyhedron $\rho_{i}$ in $\R^{n}$ for each $i$.
A {\it face of $\sigma$} is a subset $\tau\subset \sigma$ such that
$\psi_{i} (\tau \cap U_{i})$ is an intersection of $\operatorname{Im}\psi_{i}$
and a face of $\rho_{i}$ in $\R^{n}$ for each $i$.
\end{dfn}

\begin{dfn}

A {\it rational polyhedral complex} on a tropical manifold $M$ is 
 a locally finite collection $\Sigma$ of 
rational polyhedra in $M$ satisfying the following:
\begin{itemize}
\item Any face of $\sigma \in \Sigma$ belongs to $\Sigma$;
\item For $\sigma ,\tau\in \Sigma$ with $\sigma\cap \tau\neq \emptyset$,
the intersection is a face of both $\sigma$ and $\tau$. 
\end{itemize}
Polyhedra belonging to $\Sigma$ are called cells.
A rational polyhedral complex is said to be {\it pure $k$ dimensional}
if its maximal cells (with respect to inclusions) have dimension $k$.
\end{dfn}

\begin{dfn}
Let $M$ be an $n$ dimensional tropical manifold.
A {\it weight} of a pure $n-1$ dimensional polyhedral complex $\Sigma$ on $M$ is
a function $w$ from the set of $n-1$ cells of $\Sigma$ to $\Z$.
Then we call the pair
$(\Sigma, w)$ a {\it weighted rational polyhedral complex}.
\end{dfn}

\begin{dfn}\label{bal}
Let $(\Sigma, w)$ be a pure $n-1$ dimensional weighted polyhedral complex on 
an open subset of $\R^{n}$.
Fix an $n-2$ dimensional cell $P$ of $\Sigma$ and
let $F_{1},\ldots ,F_{k}$ be the $n-1$ dimensional cells adjacent to $P$.
We define $L(P)\subset \R^{n}$ to be the tangent space of $P$. 
Let $v_{1},\ldots ,v_{k}\in \Z^{n}/L(P)\cap \Z^{n}$ be the primitive outgoing vectors from the origin 
parallel to $F_{1}/L(P),\ldots ,F_{k}/L(P)$ respectively.
Then $(\Sigma ,w)$ is said to be {\it balanced at} $P$ if 
$\sum_{i=1,\ldots ,k}w(F_{i})v_{i}=0$ and
 $(\Sigma ,w)$ is said to be {\it balanced} if it is balanced at all $n-2$ dimensional faces $P$ of $\Sigma$.

Let $M$ be a tropical manifold with an atlas $\{ (U_{i},\psi_{i}) \}$ and
let $(\Sigma ', w')$ be a pure $n-1$ dimensional weighted polyhedral complex on $M$.
$(\Sigma ', w')$ is said to be {\it balanced} if its restriction to $U_{i}$ is identified with a balanced
weighted polyhedral complex via $\psi_{i}$.
\end{dfn}

\begin{exm}\label{exmdiv}
Let $f\colon \R^{n}\to \R$ be a piecewise-linear function whose slopes lie in $(\Z^{n})^{*}$
and $D(f)$ be the locus where $f$ is nonlinear on any neighborhood.
Then $D(f)$ is a (non-compact) polyhedral complex whose maximal cells have integral slopes.

For each maximal cell of $D(f)$,
we can define the multiplicity given by $f$ as follows.
For each maximal cell $F$ of $D(f)$, 
there exist two connected components $V_{1},V_{2}$ of $\R^{n}-D(f)$ on which
the restriction $f|_{V_{i}}$ is affine linear and 
the intersection of the closures of $V_{1}$ and $V_{2}$ is $F$.
Then the difference of the slopes of $f|_{V_{1}},f|_{V_{2}}$ is in $(\Z^{n})^{*}$.
Thus we can take the lattice length $l_{F}$ of this difference, that is, 
the maximal positive integer $k$ such that this difference can be divided by $k$.
We define the multiplicity of $F$ as the lattice length $l_{F}$ if $f$ is convex around $F$,
as $-l_{F}$ if $f$ is concave around $F$.
Then $D(f)$ is not only a weighted integral polyhedral complex, but also 
a balanced one.

We draw some examples of balanced weighted graph on $\R^{2}$ given by
$f_{1}=\max\{ x,y,0\} $,
$f_{2}=\max \{ 2x,y,x+y,0\}$
and
$f_{3}=\max \{ 3x,3y,2x+y,0 \}$
in Figure \ref{exm1}, \ref{exm2} and \ref{exm3}.
\begin{figure}[h]
\begin{minipage}{0.3\textwidth}
\center
\unitlength 0.1in
\begin{picture}(  8.2000,  8.0000)(  1.8000,-10.0000)
%
{\color[named]{Black}{%
\special{pn 8}%
\special{pa 1000 200}%
\special{pa 600 600}%
\special{fp}%
\special{pa 600 600}%
\special{pa 600 1000}%
\special{fp}%
\special{pa 600 600}%
\special{pa 200 600}%
\special{fp}%
}}%
\put(4.0000,-4.0000){\makebox(0,0){$y$}}%
\put(8.0000,-8.0000){\makebox(0,0){$x$}}%
\put(4.0000,-8.0000){\makebox(0,0){$0$}}%
\put(4.0000,-6.0000){\makebox(0,0){$\bf{1}$}}%
\put(6.0000,-8.0000){\makebox(0,0){$\bf{1}$}}%
\put(8.0000,-4.0000){\makebox(0,0){$\bf{1}$}}%
\end{picture}%
\caption{$D(f_{1}(x,y))$}
\label{exm1}
\end{minipage}
\begin{minipage}{0.3\textwidth}
\center
\unitlength 0.1in
\begin{picture}( 14.2000, 16.0000)(  1.8000,-18.0000)
\put(4.0000,-10.0000){\makebox(0,0){$y+1$}}%
\put(11.0000,-14.0000){\makebox(0,0){$2x$}}%
\put(4.0000,-16.0000){\makebox(0,0){$0$}}%
\put(11.0000,-4.0000){\makebox(0,0){$x+y-1$}}%
\put(8.0000,-6.0000){\makebox(0,0){$\bf{1}$}}%
\put(12.0000,-6.0000){\makebox(0,0){$\bf{1}$}}%
\put(7.0000,-12.0000){\makebox(0,0){$\bf{1}$}}%
\put(4.0000,-14.0000){\makebox(0,0){$\bf{1}$}}%
\put(6.0000,-16.0000){\makebox(0,0){$\bf{1}$}}%
%
{\color[named]{Black}{%
\special{pn 8}%
\special{pa 200 1400}%
\special{pa 600 1400}%
\special{fp}%
\special{pa 600 1400}%
\special{pa 600 1800}%
\special{fp}%
\special{pa 600 1400}%
\special{pa 800 1000}%
\special{fp}%
\special{pa 800 1000}%
\special{pa 800 200}%
\special{fp}%
\special{pa 800 1000}%
\special{pa 1600 200}%
\special{fp}%
}}%
\end{picture}%
\caption{$D(f_{2}(x,y))$}
\label{exm2}
\end{minipage}
\begin{minipage}{0.3\textwidth}
\center
\unitlength 0.1in
\begin{picture}( 19.2000, 10.0000)(  2.8000,-12.0000)
%
{\color[named]{Black}{%
\special{pn 8}%
\special{pa 400 600}%
\special{pa 600 600}%
\special{fp}%
\special{pa 600 600}%
\special{pa 600 1200}%
\special{fp}%
\special{pa 600 600}%
\special{pa 1200 400}%
\special{fp}%
\special{pa 1200 400}%
\special{pa 1800 1000}%
\special{fp}%
\special{pa 1800 1000}%
\special{pa 1800 1200}%
\special{fp}%
\special{pa 1800 1000}%
\special{pa 2200 600}%
\special{fp}%
\special{pa 1200 400}%
\special{pa 1400 200}%
\special{fp}%
}}%
\put(12.0000,-10.0000){\makebox(0,0){$x+3$}}%
\put(20.0000,-8.0000){\makebox(0,0){$\bf{1}$}}%
\put(15.0000,-7.0000){\makebox(0,0){$\bf{1}$}}%
\put(9.0000,-5.0000){\makebox(0,0){$\bf{1}$}}%
\put(6.0000,-9.0000){\makebox(0,0){$\bf{1}$}}%
\put(5.0000,-6.0000){\makebox(0,0){$\bf{3}$}}%
\put(13.0000,-3.0000){\makebox(0,0){$\bf{2}$}}%
\put(18.0000,-11.0000){\makebox(0,0){$\bf{2}$}}%
\put(20.0000,-11.0000){\makebox(0,0){$3x$}}%
\put(17.5000,-5.5000){\makebox(0,0){$2x+y+3$}}%
\put(6.0000,-4.0000){\makebox(0,0){$3y-3$}}%
\put(5.0000,-9.0000){\makebox(0,0){$0$}}%
\end{picture}%
\caption{$D(f_{3}(x,y))$}
\label{exm3}
\end{minipage}
\end{figure} 
\end{exm}

\begin{dfn}\cite[Definition 4.10]{M}
Let $X=\R^{n}/\Lambda$ be an $n$-dimensional tropical torus.

A {\it divisor} $D$ on $X$ is a balanced weighted polyhedral complex on $X$.
Any divisor can be locally written
as $D(f)$ for a rational function $f$ on an open subset of $X$.
In other words, a divisor is given by data $\{ (U_{i},f_{i}) \}$ where
$\{ U_{i} \}$ is a covering  of $X$ and $\{ f_{i}\colon U_{i}\to \R \}$ is a rational function
such that $f_{i}-f_{j}$ is a $\Z$-affine linear function on $U_{i}\cap U_{j}$.

We identify two divisors if they become the same after subdivision and deletion of
cells of weight zero.
The set of divisors on $X$ is called the divisor group of $X$,
denoted $\operatorname{Div}(X)$.
Here the sum of divisors on $X$ is defined by
a common subdivision and the sum of weight functions (restricted to the subdivision), 
and the inverse of a divisor is defined by a divisor whose support is the original support
and whose weight function is $-1$ times the original weight.
\end{dfn}

In terms of local data,
the sum of divisors $\{ (U_{i},f_{i}) \}$ and $\{ (V_{i},g_{i}) \}$ is 
given by $\{ (U_{i}\cap V_{j},f_{i}+g_{j}) \}$
and the inverse of a divisor $\{ (U_{i},f_{i}) \}$ is given by $\{ (U_{i},-f_{i}) \}$.

As in the classical case, divisors give rise to line bundles.
Let $D$ be a divisor on $X$ and  $\{ (U_{i},f_{i}) \}$ be the local data of $D$.
Then we can define the line bundle $\mathcal{O}(D)$ associated to $D$
to be the line bundle on $X$ whose transition functions  are given by $\varphi_{ij}=f_{i}-f_{j}$.
It is clear that the above map from the divisor group to the Picard group is a group homomorphism.

Two divisors $D_{1}, D_{2}$ on $X$ are {\it linearly equivalent} if
the difference $D_{1}-D_{2}$ is given by a global rational function,
or in other words, $\mathcal{O}(D_{1})$ and $\mathcal{O}(D_{2})$ are isomorphic.

\subsection{Theta functions}
The content of this section is basically due to Mikhalkin-Zharkov \cite{MZ}.

Let $X=\R ^{n}/\Lambda $ be a tropical torus.
The sheaf $\mathcal{T}_{\Z}^{*}$ is defined by the following exact sequence of sheaves of abelian groups:
\begin{equation}\label{exactsq}
 0\to \underline{\R} \to {\rm Aff}_{\Z} \to \mathcal{T}_{\Z}^{*}\to 0
\end{equation}
where $\underline{\R}$ denote the constant sheaf on $X$ whose stalks are $\R$. 
Then $\mathcal{T}_{\Z}^{*} \cong \underline{(\Z ^{n})^{*}}$.

Taking the first sheaf cohomology,
we get {\it the Chern class map }
\[ c_{1}: H^{1}(X,{\rm Aff}_{\Z}) \to H^{1}(X,\mathcal{T}_{\Z}^{*}). \]
We obtain the following proposition
by analyzing the coboundary map of the long exact sequence.

\begin{prop}[Mikhalkin-Zharkov \cite{MZ} Section 5.1]
The image of $c_{1}$ in $H^{1}(X,\mathcal{T}_{\Z}^{*})\cong 
\Lambda ^{*}\otimes (\Z ^{n})^{*}$ is the set of
elements of $\Lambda ^{*}\otimes (\Z ^{n})^{*}$ which can be naturally
extended to symmetric bilinear forms on $\R^{n}$.
Here $\Lambda ^{*}:=\operatorname{Hom}_{\Z}(\Lambda, \Z)$.
\end{prop}

\begin{dfn}
If a class $[c]\in {\rm Im}\,c_{1}$ is positive definite as a symmetric form
on $\R^{n}$, it is called a {\it polarization }.
The {\it degree of a polarization} is the cardinal of the cokernel
$\Cok ( [c] )= (\Z^{n})^{*}/[c](\Lambda)$, where we identify $[c]\in \Lambda ^{*}\otimes (\Z ^{n})^{*}$
with the corresponding element of $\operatorname{Hom}_{\Z}(\Lambda, (\Z^n)^{*})$.
If a polarization $[c]\in\operatorname{Hom}_{\Z}(\Lambda, (\Z^n)^{*})$ is an isomorphism, 
$[c]$ is said to be {\it principal}.
\end{dfn}
\begin{dfn}
A {\it tropical abelian variety} is a tropical torus which has a polarization. 
A pair $(X, [c])$ of tropical abelian variety and its polarization is called
{\it polarized tropical abelian variety}.
\end{dfn}

In this section we consider tropical tori, which are not necessarily tropical abelian varieties.
We will consider abelian surfaces to calculate the intersection numbers
in Section \ref{intersection number}.

Let $X=\R /\Lambda$ be a tropical torus and
let $L$ be a tropical line bundle on $X$.
We set $Q:=c_{1}(L)\in \Lambda^{*}\otimes (\Z^{n})^{*}$.


Since the sheaf of invertible regular functions ${\rm Aff}_{\Z}$  is constant on $\R^{n}$, 
the first cohomology $H^{1}(\R^{n}, {\rm Aff}_{\Z})$ vanishes.
Thus the pull back $p^{*}L $ of $L$ by the projection $p\colon \R^{n}\to \R^{n}/\Lambda$
is isomorphic to the trivial bundle $\R^{n}\times \T$.
Fix this isomorphism once. 

The lattice $\Lambda$ acts on $\R^{n}$ by translation.
We lift this action to an action on $p^{*}L$ by
\[  \lambda \cdot (x,\xi) := (x+\lambda ,\xi), \ \lambda\in \Lambda, x\in \R^{n},\xi\in L. \]
By the isomorphism $p^{*}L \cong \R^{n}\times\T$ we get a $\Lambda$-action on $\R^{n}\times \T$.
This action can be described as
\[ \lambda\cdot (x,t) := (x+\lambda, t+Q_{\R}(\lambda ,x)+\beta_{L}(\lambda )) \]
for some $\beta_{L}\colon \Lambda \to \R$.
Here $Q_{\R}\in (\R^{n})^{*}\otimes (\R^{n})^{*}$ is the extension of $Q$ on $\R^{n}$.
By the associativity of the action, $\beta_{L}$ satisfies 
$\beta_{L} (\lambda_{1}+\lambda_{2})=
\beta_{L} (\lambda_{1})+\beta_{L} (\lambda_{2})+Q_{\R}(\lambda _{1},\lambda _{2})$
for any $\lambda_{1},\lambda_{2} \in \Lambda$. 
Thus $\beta_{L} (\lambda)-\frac{1}{2}Q_{\R}(\lambda, \lambda )$ is linear in $\lambda$.
We denote this linear function by $\alpha_{L} (\lambda)$.

\begin{dfn}
For a linear function $\alpha \colon \R^{n} \to \R$ and
$Q\in (\Lambda)^{*}\otimes (\Z^{n})^{*}$ which extends to a symmetric form on $\R^{n}$,
We define a $\Lambda$-action on $\R^{n}\times \T$ by
\[ 
\lambda \cdot (x,t)
:= (x+\lambda ,t+Q_{\R}(\lambda ,x)+\frac{1}{2}Q_{\R}(\lambda, \lambda)+\alpha (\lambda))
\]
for all $\lambda \in \Lambda, x\in \R^{n}, t\in \T$.
We can easily check that this defines a $\Lambda$-action.
Then we define the line bundle $L(Q,\alpha )$ on $X=\R^{n}/\Lambda$ by $(\R^{n}\times \T)/\Lambda$.
The local trivializations of $L(Q,\alpha )$ are induced from 
the given trivialization of the trivial bundle $\R^{n}\times \T$.
\end{dfn}

Now we get the following proposition.
\begin{prop}\label{LQA}
The map $L(\cdot ,\cdot ): \operatorname{Im}c_{1}\times (\R^{n})^{*}\to H^{1}(X, {\rm Aff}_{\Z} )$
given by $(Q, \alpha)\mapsto L(Q, \alpha )$ is a surjective group homomorphism.
Moreover, the kernel of this homomorphism is $0\times (\Z^{n})^{*}$, that is,
$L(Q,\alpha)\cong L(Q,\alpha +\gamma)$ for any $\gamma\in (\Z^{n})^{*}$.
\end{prop}
\begin{proof}
The long exact sequence given by (\ref{exactsq}) induces the following short exact sequence;
\[
 0\to H^{1}(X, \underline{\R})/H^{0}(X,\mathcal{T}_{\Z}^{*})= (\R^{n})^{*}/(\Z^{n})^{*}
\to H^{1}(X, {\rm Aff}_{\Z}) \to \operatorname{Im}c_{1}\to 0.
\]
This splits by $L(\cdot ,0)\colon \operatorname{Im}c_{1}\to H^{1}(X, {\rm Aff}_{\Z})$,
thus $H^{1}(X, {\rm Aff}_{\Z})$ is isomorphic to $\operatorname{Im}c_{1}\times ((\R^{n})^{*}/(\Z^{n})^{*})$.
Then the map $L(\cdot ,\cdot )$ is identified with the quotient map
$ \operatorname{Im}c_{1}\times (\R^{n})^{*}\to \operatorname{Im}c_{1}\times ((\R^{n})^{*}/(\Z^{n})^{*})$.
The kernel of this homomorphism is clearly  $0\times (\Z^{n})^{*}$.
\end{proof}

Any global section of $L(Q,\alpha )$ is identified with a regular function
$\Theta\colon \R^{n}\to \R$ satisfying the following quasi-periodicity condition:
\[ \Theta (x+\lambda )=
\Theta (x)+Q_{\R}(\lambda ,x)+\beta (\lambda ) 
,\ \lambda\in \Lambda ,\ x\in \R^{n} \]
with $\beta (\lambda )$ is equal to $\alpha (\lambda )+\frac{1}{2}Q(\lambda ,\lambda )$.
The regularity implies that $\Theta$ is convex and
the quasi-periodicity implies the fact that
if $b\in (\Z^{n})^{*}$ is a slope of $\Theta$,
then $b+Q(\lambda , \cdot )$ is also a slope of $\Theta$ for any $\lambda\in \Lambda$.

By the above argument, the tropical module $\Gamma  (X,L(Q, \alpha ))$
of regular sections of $L(Q, \alpha )$ is identified with the set
\[
\left\{ \Theta \colon \R^{n}\to \R \ \middle|\ 
\begin{aligned}
&\Theta (x+\lambda )=\Theta (x)+Q_{\R}(\lambda ,x)+\beta (\lambda ),
\ \lambda\in \Lambda ,\ x\in \R^{n} \\
~&\hspace{0mm} \text{$\Theta (x)$ is a regular function.}
\end{aligned} \right\} \cup \{ -\infty \}.
\]

An element of this set which is not $\{  -\infty\}$ is called a {\it (tropical) theta function}.

\begin{rmk}
If $Q$ has a negative eigenvalue, $\beta (\lambda )$ diverges to $-\infty$ quadratically
in the direction of the eigenvector.
Then any function satisfying the quasi-periodicity is not convex,
thus no theta function exists 
\end{rmk}

For $r\in \R^{n}$, we define the translation
$\iota _{r}: \R^{n}/\Lambda \to \R^{n}/\Lambda$ by $x\mapsto x+r$.
The pulled back line bundle $\iota _{r}^{*}L(Q, \alpha )$
is given by the quotient of $\R^{n}\times \T$ by the following $\Lambda$-action.
\[
\lambda \cdot (x-r,t):=
(x-r+\lambda ,t+Q_{\R}(\lambda ,x)+\frac{1}{2}Q_{\R}(\lambda, \lambda)+\alpha (\lambda)).
\]
That is,
\begin{align*}
 \lambda \cdot (x,t)&= (x+\lambda ,t+Q_{\R}(\lambda ,x+r)+\frac{1}{2}Q_{\R}(\lambda, \lambda)+\alpha (\lambda))\\
 &=(x+\lambda ,
 t+Q_{\R}(\lambda ,x)+\frac{1}{2}Q_{\R}(\lambda, \lambda)+\alpha (\lambda)+ Q_{\R}(\lambda, r)).
\end{align*}
Thus we get $\iota_{r}^{*}L(Q, \alpha )=L(Q, \alpha +Q_{\R}(\cdot ,r))$.

\subsection{The space of theta function: positive definite case}
Given a bilinear form $Q\in \Lambda^{*}\otimes (\Z^{n})^{*}$,
let $q$ denote the linear map $\Lambda \to (\Z^{n})^{*}$, $\lambda \mapsto Q(\lambda ,\cdot )$
and let $q_{\R}$ denote the linear map $\R^{n}\to (\R^{n})^{*}$ extending $q$.

In this section, we consider the case where
$Q$ is positive definite as a symmetric bilinear form.
Then for any $\alpha\in (\R^{n})^{*}$, there exists an element $r\in \R^{n}$ such that $\alpha =q_{\R}(r)$.

Let $\Theta$ be a theta function for $L(Q,q_{\R}(r))$ and
consider the Legendre transform  of $\Theta$,
\[ \widehat{\Theta}(a)=
\max_{x\in \R^{n}} \left\{a\cdot x-\Theta (x) \right\}, a\in (\Z^{n})^{*}.\]
The maximum exists since $\Theta$ diverges to $+\infty$ quadratically
as $|x|\to \infty$; this follows by the positive definiteness of $Q$ and quasi-periodicity of $\Theta$.
The Legendre transform $\widehat{\Theta}\colon (\Z^{n})^{*}\to \R$
is convex, that is, 
$\widehat{\Theta}(a)\leq \sum t_{i}\widehat{\Theta}(a_{i})$
for $a, a_{i}\in (\Z^{n})^{*}$ and $0\leq t_{i}\leq 1$ satisfying  $\sum t_{i}=1$, $a=\sum t_{i}a_{i}$.

The Legendre transform $\widehat{\Theta}$ satisfies the following quasi-periodicity:
\begin{align*}
\widehat{\Theta}(a+q(-\lambda ))
&= \max_{x\in \R^{n}}\left\{(a+q(-\lambda ))\cdot x-\Theta (x) \right\}\\
&=\max_{x\in \R^{n}}\left\{a\cdot x-\Theta (x+\lambda )
+\frac{1}{2}Q_{\R}(\lambda ,\lambda)+q_{\R}(r)(\lambda ) \right\}\\
&=\max_{x\in \R^{n}}\left\{a\cdot (x-\lambda )-\Theta (x)
+\frac{1}{2}Q_{\R}(\lambda ,\lambda)+q_{\R}(r)(\lambda )\right\}\\
&=\max_{x\in \R^{n}} \left\{a\cdot x -\Theta (x) \right\}
-a\cdot \lambda +\frac{1}{2}Q_{\R}(\lambda ,\lambda)+Q_{\R}(\lambda ,r)\\
&=\widehat{\Theta}(a)-a\cdot \lambda +\frac{1}{2}Q_{\R}(\lambda ,\lambda)+Q_{\R}(\lambda ,r)
\end{align*}
for all $a\in (\Z^{n})^{*}$, $\lambda\in\Lambda$.
Here the second equality follows by the quasi-periodicity of $\Theta$ and the third equality
follows by replacing $x$ with $x-\lambda$.
Moreover, the Legendre transformation gives one-to-one correspondence between
the set of theta functions for $L(Q,q_{\R}(r))$ and the following set:
\begin{eqnarray}\label{Legendre}
\left\{ \eta \colon (\Z^{n})^{*}\to \R \ \middle|\ 
\begin{aligned}
&\text{$\eta $ is convex and for all }\lambda\in \Lambda ,\ a\in (\Z^{n})^{*}, \\
~&\hspace{0mm}
\eta (a+q(\lambda ))=\eta (a)+a\cdot\lambda +\frac{1}{2}Q_{\R}(\lambda ,\lambda)-Q_{\R}(r, \lambda) 
\end{aligned} \right\}.
\end{eqnarray}

We choose a subset $B\subset (\Z^{n})^{*}$ such that the natural map
$B\to (\Z^{n})^{*}/q(\Lambda )$ is a bijection.
Then $\widehat{\Theta}$ is determined by the values at $b\in B$ by the quasi-periodicity.
Because of the convexity of the theta function $\Theta$, the inverse Legendre transform of 
$\widehat{\Theta}$ is $\Theta$. Thus $\Theta(x)$ can be written as follows:
\begin{align*}
\Theta (x)
&=\max_{a\in (\Z^{n})^{*}} \left\{a\cdot x-\widehat{\Theta}(a) \right\}\\
&=\max_{b\in B}\max_{\lambda \in \Lambda }\left\{(b+q(-\lambda ))\cdot x
-\widehat{\Theta}(b+q(-\lambda ))\right\}\\
&=\max_{b\in B}\max_{\lambda \in \Lambda }\left\{(b+q(-\lambda ))\cdot x
-\widehat{\Theta}(b)+b\cdot \lambda -\frac{1}{2}Q_{\R}(\lambda ,\lambda )
-Q_{\R}(\lambda ,r)\right\} \\
&=\max_{b\in B}\left\{\Theta_{b}(x)+\frac{1}{2}Q_{\R}
(q_{\R}^{-1}(b)-r,q_{\R}^{-1}(b)-r)-\widehat{\Theta}(b)\right\},
\end{align*}
where  
\begin{equation*}
\begin{split}
\Theta_{b}(x)
&:=\max_{\lambda \in \Lambda }\left\{(b+q(\lambda ))\cdot x-\frac{1}{2}Q_{\R}
(\lambda +q_{\R}^{-1}(b)-r ,\lambda +q_{\R}^{-1}(b)-r)\right\}. 
\end{split}
\end{equation*}

By the above argument, we get the fact that theta functions for $L(Q,q_{\R}(r) )$ exist and
every theta function for $L(Q,q_{\R}(r) )$ can be written in the form:
\[
\max_{b\in B}\{ \Theta_{b}(x)+s_{b} \}
\]
for some $\{  s_{b} \}\in \T^{B}$.
The set of slopes of $\Theta_{b}$ is $b+q(\Lambda )$
and the sets $b+q(\Lambda ),b\in B$ are pairwise disjoint.
Thus $\Theta_{b}$ cannot be written as any tropical linear combination
 of elements of $H^{0} (X,L(Q,q_{\R}(r) ))\backslash \T \Theta_{b}$.
Consequently any minimal generating set of the $\T$-module $\Gamma (X,L(Q,q_{\R}(r) ))$
 is of the form $\{ \Theta_{b}+s_{b} \}_{b\in B}$; in particular, 
the minimal number of generators is $|B|$.
Let $\pi\colon \T^{B}\to H^{0}(X,L(Q,q_{\R}(r) ))$ be the surjective tropical homomorphism
$\displaystyle (s_{b})_{b\in B}\mapsto \max_{b\in B}\{ \Theta_{b}+s_{b} \}$.
 \begin{dfn}\label{inj}
We define the map $\varphi \colon \Gamma (X,L(Q,q_{\R}(r) ))\to \T^{B}$ by
\[ \varphi (\Theta )=\left( \varphi ^{b}(\Theta ) \right)_{b\in B}, \]
where
\begin{align*}
\varphi ^{b}(\Theta )&:=\max \left\{ s_{b}\in \T \ \middle|\ 
\Theta (x)\geq \Theta_{b}(x)+s_{b}\ 
\forall x\in \R^{n} \right\} \\
&= \min_{x\in \R^{n}}\left\{ \Theta (x)-\Theta_{b}(x) \right\} .
\end{align*}
\end{dfn}

The composition $\pi \circ \varphi$
is the identity map of $\Gamma (X,L(Q,q_{\R}(r) ))$.
Thus the map $\varphi $ is injective and we can induce a topology of
$\Gamma (X,L(Q,q_{\R}(r) ))$ by identifying it with
the subspace $\varphi (\Gamma (X,L(Q,q_{\R}(r) )) )\subset \T^{B}$.

\begin{rmk}
The map $\varphi \colon \Gamma (X,L(Q,q_{\R}(r) ))\to \T^{B}$
is not a map of tropical modules.
This map measures  the ``meaningful'' coefficients $s_{b}$ of $\Theta_{b}$ appearing in the
expression $\Theta (x)=\max_{b\in B}\{ \Theta_{b}(x)+s_{b} \}$.
\end{rmk}

\begin{prop}\label{prop}
Let $r(b)$ denote $\frac{1}{2}Q_{\R}(q_{\R}^{-1}(b)-r, q_{\R}^{-1}(b)-r)$. Then
\[ \varphi^{b}(\Theta )=-\widehat{\Theta }(b)+r(b)\]
\end{prop}
\begin{proof}
Recall the set of slopes of a generator
\[ \Theta_{b}(x)
=\max_{\lambda \in \Lambda }\left\{(b+q(\lambda ))\cdot x-\frac{1}{2}Q_{\R}
(\lambda +q_{\R}^{-1}(b)-r ,\lambda +q_{\R}^{-1}(b)-r)\right\}\]
is exactly $\left\{ b+q(\lambda ) \ \middle|\  \lambda \in \Lambda \right\}$.
Let $D^{b}_{\lambda}$ be the maximal closed domain on which
$\Theta_{b}(x)$ is an affine linear function of slope $b+q(\lambda )$.
By the quasi-periodicity, we get
\[ D^{b}_{\lambda_{1}+\lambda_{2}}=
\lambda_{2}+D^{b}_{\lambda_{1}}\]
for $\lambda_{1},\lambda_{2}\in \Lambda$
and thus the image of $D^{b}_{\lambda}$ by the quotient map $\R^{n}\to X=\R^{n}/\Lambda$
is the whole of $X$.
It is clear that
$\Theta_{b}(x)=b\cdot x-r(b)$ holds for any point $x$ in the fundamental domain $D^{b}_{0}$ of $X$.

Since $\Theta (x)-\Theta _{b}(x)$ is $\Lambda$-periodic by the quasi-periodicity of $\Theta(x)$,
\begin{align*}
\varphi^{b}(\Theta )&=
\min_{x\in \R^{n}}\{ \Theta (x)-\Theta _{b}(x) \}\\
&=\min_{x\in D_{0}^{b}}\{ \Theta (x)-\Theta _{b}(x) \}\\
&=\min_{x\in D_{0}^{b}}\{ \Theta (x)-(b\cdot x-r(b)) \}\\
&=\min_{x\in \R^{n}}\{ \Theta (x)-(b\cdot x-r(b)) \} 
\text{\ \ \ (by the convexity of $\Theta_{b}$)}
\\
&=-\max_{x\in \R^{n}}\{ b\cdot x -\Theta (x) \} +r(b) \\
&=-\widehat{\Theta}(b)+r(b).
\end{align*}
\end{proof}

\begin{exm}
Let us consider a tropical torus $\R/\Z$ and a line bundle $L(3,0)$.
We take $B=\{ 0,1,2 \}$ and consider the generators of $\Gamma (\R/\Z ,L(3,0))$,
\[
\Theta_{b}(x)=\max_{n\in \Z}\left\{  
(3n+b)x-\frac{3}{2}\left( \frac{3n+b}{3}\right) ^{2}
\right\},\ b\in B.
\]

Let $\varphi \colon \Gamma (\R/\Z ,L(3,0))\to \T^{3}$ be the injection
as in Definition \ref{inj}.
Let us calculate the set
$\{ (0,r,s) \mid r,s\in \R \}\cap \varphi (\Gamma (\R/\Z ,L(3,0)))$; this is
identified with the tropical projectivization of $\varphi (\Gamma (\R/\Z ,L(3,0)))$.

The strategy of calculation is as follows.
First we calculate the region of possible values of $\varphi^1(\Theta)$
under the condition that $\varphi^0(\Theta) = 0$.
We can easily see that $\varphi^1(\Theta) $ ranges over a closed interval
whose lower or upper bounds are given by the minimum or maximal of the function:
\[\Theta_{0}(x)-\Theta_{1}(x). \]
By computing these values, we get the range of $\varphi^{1}(\Theta )$.
It is $\left[ -\dfrac{1}{3},\dfrac{1}{3}\right]$.

Next, we fix the value of $\varphi^{1}(\Theta )$ and denote the value by $r$.
We calculate the upper and lower bound of $\varphi^{2}(\Theta )$.
The lower bound is the minimal value of the difference of theta functions
\[ \max\{ \Theta_{0}(x), \Theta_{1}(x)+r  \}
-\Theta_{2}(x).\]
The upper bound is
the supremum of $s$ such that 
the terms $\Theta_{0}$ and $\Theta_{1}+r$ cannot be removed from the equation
$\Theta = \max (\Theta_{0}, \Theta_{1}+r, \Theta_{2}+s)$, that is, 
the maximal $s$ satisfying the following two equations: 
\begin{align*}
0&\geq\min_{x\in \R^{2}}\left\{ \max \left\{ \Theta_{1}(x)+r,
\Theta_{2}(x)+s \right\} -\Theta_{0}(x)\right\} ; \\
0&\geq\min_{x\in \R^{2}}\left\{ \max \left\{ \Theta_{0}(x)
,\Theta_{2}(x)+s \right\} -(\Theta_{1}+r)\right\} .
\end{align*}
By computing these values, we get the range of $\varphi^{2}(\Theta )$.
It is \[  \frac{1}{2}r+\frac{1}{6}\leq \varphi^{2}(\Theta )\leq \min \{ 2r,-r \} +\frac{1}{3}. \]
Thus the slice $\{ (0,r,s) \mid r,s\in \R \}\cap \varphi (\Gamma (\R/\Z ,L(3,0)))$
 is given by
\[ \left\{
(r,s)\in \R^{2}\middle|\ s\leq 2r+\frac{1}{3}, r\leq 2s+\frac{1}{3}, r+s\leq \frac{1}{3}
\right\}. \]
This is a $2$-dimensional simplex as in Figure \ref{figrei2}
and  each of generators corresponds to a vertex of this simplex.
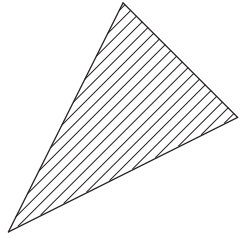
\begin{figure}[h]
\center
{\unitlength 0.1in%
\begin{picture}(12.0000,12.0000)(2.0000,-14.0000)%
%
\special{pn 8}%
\special{pa 800 200}%
\special{pa 1400 800}%
\special{fp}%
\special{pa 1400 800}%
\special{pa 200 1400}%
\special{fp}%
\special{pa 200 1400}%
\special{pa 800 200}%
\special{fp}%
%
\special{pn 4}%
\special{pa 1050 450}%
\special{pa 300 1200}%
\special{fp}%
\special{pa 1080 480}%
\special{pa 240 1320}%
\special{fp}%
\special{pa 1110 510}%
\special{pa 250 1370}%
\special{fp}%
\special{pa 1140 540}%
\special{pa 370 1310}%
\special{fp}%
\special{pa 1170 570}%
\special{pa 490 1250}%
\special{fp}%
\special{pa 1200 600}%
\special{pa 610 1190}%
\special{fp}%
\special{pa 1230 630}%
\special{pa 730 1130}%
\special{fp}%
\special{pa 1260 660}%
\special{pa 850 1070}%
\special{fp}%
\special{pa 1290 690}%
\special{pa 970 1010}%
\special{fp}%
\special{pa 1320 720}%
\special{pa 1090 950}%
\special{fp}%
\special{pa 1350 750}%
\special{pa 1210 890}%
\special{fp}%
\special{pa 1380 780}%
\special{pa 1330 830}%
\special{fp}%
\special{pa 1020 420}%
\special{pa 360 1080}%
\special{fp}%
\special{pa 990 390}%
\special{pa 420 960}%
\special{fp}%
\special{pa 960 360}%
\special{pa 480 840}%
\special{fp}%
\special{pa 930 330}%
\special{pa 540 720}%
\special{fp}%
\special{pa 900 300}%
\special{pa 600 600}%
\special{fp}%
\special{pa 870 270}%
\special{pa 660 480}%
\special{fp}%
\special{pa 840 240}%
\special{pa 720 360}%
\special{fp}%
\special{pa 810 210}%
\special{pa 780 240}%
\special{fp}%
\end{picture}}%
\caption{the slice $\{ (0,r,s) \mid r,s\in \R \}\cap \varphi (\Gamma (\R/\Z ,L(3,0)))$}
\label{figrei2}
\end{figure}
\end{exm}

\begin{exm}\label{Q}
In the previous example, the projectivization of $\Gamma (X,L)$ is a simplex
and its vertices correspond to generators.
However it is not always true. We give a counter-example.

Let $q_{\R}\colon \R^{2}\to (\R^{2})^{*}$ be the linear map given by
\footnotesize
$\left( \begin{array}{cc} 2&1\\ 1&2\end{array}\right)$
\normalsize
and let $\Lambda$ be the lattice $q_{\R}^{-1}(2(\Z^{2})^{*})$.
These data define an element $Q\in \Lambda^{*}\otimes (\Z^2)^{*}$ given by $q_{\R}$
and a tropical abelian surface $X:=(\R^{2}/\Lambda ,Q)$.

Let us consider the tropical line bundle $L=L(Q,0 )$ on $X$.
We take $B=\{ b_{ij} \}_{i,j\in \{ 0,1\}}$ with $b_{ij}={}^t\! ( i, j) $ 
and consider the generators of $\Gamma (X,L)$,
\[ \Theta_{b_{ij}}(x)=
\max_{k,l\in \Z}\left\{ (2k+i)x+(2l+j)y-
\frac{1}{3}\left( (2k+i)^{2}-(2k+i)(2l+j)+(2l+j)^{2}\right) \right\} . \]

Define the injection
$\varphi \colon \Gamma (X,L)\to \T^{4}$ as
\[ \varphi (\Theta )=(\varphi^{00}(\Theta ),\varphi^{01}(\Theta ),
\varphi^{10}(\Theta ),\varphi^{11}(\Theta )) \]
where $\varphi^{ij}(\Theta ):=\varphi^{b_{ij}}(\Theta ) $ is as in Definition \ref{inj}.
By computing as in the previous example,
the slice $\{ (0,r,s,t) \mid r,s,t\in \R \}\cap \varphi (\Gamma (X,L))$
is given by
\begin{equation*}
\left\{ (0, r,s,t) \ \middle|\  
\begin{split}
~&-\frac{1}{3}\leq r\leq \frac{1}{3},\ 
\max \{ 0,r \} -\frac{1}{3} \leq s \leq \min \{ 0,r \} +\frac{1}{3},\\
&\max \{ 0,r,s \} -\frac{1}{3} \leq t \leq \min \{ 0,r,s \} +\frac{1}{3}
\end{split}
\right\}. 
\end{equation*}
This region can also be written as 
\[ \bigcup_{ -\frac{1}{6}\leq a\leq \frac{1}{6}}
\left\{ (r,s,t)+(a,a,a) \ \middle|\  
|r|\leq \frac{1}{6},\ |s|\leq \frac{1}{6},\ |t|\leq \frac{1}{6},\ 
\right\} . \]
This is the convex polyhedron that has $14$ vertices as in Figure \ref{figrei}.

The $4$ generators of the tropical module $\Gamma (X,L)$ correspond to
vertices of this polyhedron, but this polyhedron has more vertices.

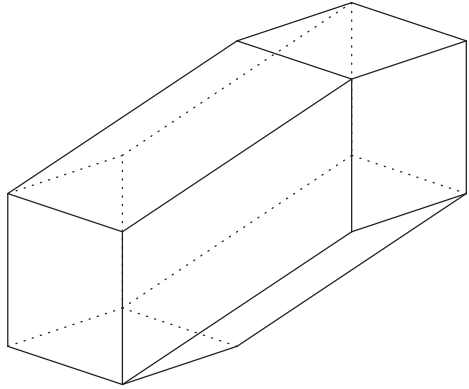
\begin{figure}[h]
\center
\unitlength 0.1in
\begin{picture}( 24.0000, 20.0000)(  2.0000,-22.0000)
%
{\color[named]{Black}{%
\special{pn 8}%
\special{pa 200 1200}%
\special{pa 200 2000}%
\special{fp}%
\special{pa 200 2000}%
\special{pa 800 2200}%
\special{fp}%
\special{pa 800 2200}%
\special{pa 1400 2000}%
\special{fp}%
\special{pa 200 1200}%
\special{pa 800 1400}%
\special{fp}%
\special{pa 2000 1400}%
\special{pa 800 2200}%
\special{fp}%
\special{pa 2000 1400}%
\special{pa 2600 1200}%
\special{fp}%
\special{pa 2600 1200}%
\special{pa 1400 2000}%
\special{fp}%
\special{pa 2600 1200}%
\special{pa 2600 400}%
\special{fp}%
\special{pa 2600 400}%
\special{pa 2000 200}%
\special{fp}%
\special{pa 2000 200}%
\special{pa 1400 400}%
\special{fp}%
\special{pa 1400 400}%
\special{pa 200 1200}%
\special{fp}%
\special{pa 1400 400}%
\special{pa 2000 600}%
\special{fp}%
\special{pa 2000 600}%
\special{pa 2600 400}%
\special{fp}%
\special{pa 2000 600}%
\special{pa 2000 1400}%
\special{fp}%
\special{pa 2000 600}%
\special{pa 800 1400}%
\special{fp}%
\special{pa 800 1400}%
\special{pa 800 2200}%
\special{fp}%
}}%
%
{\color[named]{Black}{%
\special{pn 8}%
\special{pa 200 2000}%
\special{pa 800 1800}%
\special{dt 0.045}%
\special{pa 2000 200}%
\special{pa 2000 1000}%
\special{dt 0.045}%
\special{pa 2000 1000}%
\special{pa 2600 1200}%
\special{dt 0.045}%
\special{pa 2000 1000}%
\special{pa 800 1800}%
\special{dt 0.045}%
\special{pa 1400 2000}%
\special{pa 800 1800}%
\special{dt 0.045}%
\special{pa 800 1800}%
\special{pa 800 1000}%
\special{dt 0.045}%
\special{pa 800 1000}%
\special{pa 200 1200}%
\special{dt 0.045}%
\special{pa 800 1000}%
\special{pa 2000 200}%
\special{dt 0.045}%
}}%
\end{picture}%
\caption{the slice $\{ (0,r,s,t) \mid r,s,t\in \R \}\cap \varphi (\Gamma (X,L))$}
\label{figrei}
\end{figure}
\end{exm}

\begin{lem}\label{convexPL}
Let $D$ be a compact convex polyhedron in $\R^{n}$
and $f\colon D\times \R^{m}\to \R$ be 
a convex piecewise-linear function with finitely many slopes\footnote
{A convex piecewise-linear function means 
a function which is locally written as the maximum of finitely many affine linear functions.
It has finitely many slopes if and only if
it can be written as the maximum of finitely many affine linear functions. }.

We define the function $g\colon \R^{m}\to \R$ as
\[ g(y):=\min_{x\in D}\{ f(x,y) \} . \]
Then $g$ is a convex piecewise-linear function with finitely many slopes.
\end{lem}

\begin{proof}
Let $p\colon D\times \R^{n}\times \R\to \R^{n}\times \R$ be the projection and
 $\Delta_{f}$ be the upper convex hull of graph of $f$, that is,
\[ \Delta_{f}:= \left\{ (x,y,r)\in D\times \R^{n}\times \R \ \middle| \  r\geq f(x,y) \right\}. \]
Since the set of slopes of $f$ is finite, $\Delta_{f}$ has only finitely many faces. 
Thus $\Delta_{f}$ is a convex polyhedron
and the image of the projection $p( \Delta_{f})$ is also a convex polyhedron.
On the other hand $p( \Delta_{f})$ is described as
\begin{align*}
p( \Delta_{f})&= \left\{ (y,r)\in \R^{n}\times \R \ \middle| \  r\geq f(x,y) \text{ for all } x\in D \right\}\\
&=\left\{ (y,r)\in \R^{n}\times \R \ \middle| \  r\geq \min_{x\in D}\{ f(x,y)\} \right\}.
\end{align*}
That is, $p( \Delta_{f})$ is the upper convex hull of the graph of $g$.
Therefore $g$ is a convex piecewise-linear function.
Since the upper convex hull $p( \Delta_{f})$ of $g$ has only finitely many faces, 
the convex piecewise-linear function $g$ has only finitely many slopes.
\end{proof}

\begin{thm}\label{thm1}
Let $L$ be a tropical line bundle on a tropical abelian variety $X=\R^{n}/\Lambda$
such that $Q=c_{1}(L)$ is positive definite.
Then $\Gamma (X,L)$ is generated by $|B|=\left| \operatorname{Cok}\,q\right|$ elements 
as a $\T$-module.
Moreover, $\Gamma (X,L)$ is identified with
an $|B|$-dimensional convex polyhedron in $\T^{B}$ via $\varphi$ of Definition \ref{inj} 
and its projectivization is identified with a compact polyhedron in $\R^{|B|-1}$.
\end{thm}

\begin{proof}
We can assume $L=L(Q,q_{\R}(r) )$ for some $r\in \R^{*}$.
The former is already shown. Thus we show the latter.
The image $\operatorname{Im}\varphi$ is equal to the subset
\[
\left\{ (t_{b})_{b\in B} \ \middle|\ 
\varphi^{b}\left( \max_{b'\in B}\{ \Theta_{b'}(x)+t_{b'} \}\right) =t_{b} 
\ \forall b\in B
 \right\}. \]

The condition 
$\displaystyle \varphi^{b}\left( \max_{b'\in B}\{ \Theta_{b'}(x)+t_{b'} \}\right) =t_{b}$
is equivalent to 
\begin{align*}
t_{b}&\geq \min_{x\in \R^{n}}\left\{ \max_{b'\in B\setminus \{ b\}}
\left( \Theta_{b'}(x)+t_{b'}\right)-\Theta_{b}(x) 
 \right\}\\
&=\min_{x\in D_{0}^{b}}\left\{ \max_{b'\in B\setminus \{ b\}}
\left( \Theta_{b'}(x)+t_{b'}\right)-\Theta_{b}(x) 
 \right\} \text{\ \ (by the quasi-periodicity)} \\
&=\min_{x\in D_{0}^{b}}\left\{ \max_{b'\in B\setminus \{ b\}}
\left( \Theta_{b'}(x)+t_{b'}\right)-(b\cdot x -r(b)) 
 \right\},
\end{align*}
where $D_{0}^{b}$ is the region introduced in the proof of Proposition \ref{prop}.

By Lemma \ref{convexPL},
the right-hand side of the above inequality is a convex piecewise-linear function
of $t_{b'},\ b'\in B\setminus \{ b\}$ with only finitely many slopes.
Thus $\operatorname{Im}\varphi $ is a convex polyhedron.

Let $\Xi := \max_{b\in B}\left\{ \Theta_{b}(x)\right\}$ and 
let $\widetilde{\Lambda} \subset \R^{n}$ be the lattice $q_{\R}^{-1}((\Z^{2})^{*})$, 
which is an overlattice of $\Lambda$.
We can check that
$\Xi$ satisfies the quasi-periodicity
\[ \Xi (x+\eta )=
\Xi (x)+Q_{\R}(\eta ,x)+\frac{1}{2}Q_{\R}(\eta, \eta )+q_{\R }(r)
,\ \eta\in \widetilde{\Lambda} ,\ x\in \R^{n} \]
and we get that
the set of slopes of $\Xi$ is exactly $(\Z^{n})^{*}$.
We can check that $\varphi (\Xi)=0$,
thus the image $\operatorname{Im}\varphi$ contains the origin $(0,\ldots ,0)$.
Moreover, 
the image $\operatorname{Im}\varphi$ contains
$|B|$ points
$(\varepsilon ,0,\ldots , 0),\ldots ,(0,\ldots ,0, \varepsilon )$
for a sufficiently small positive number $\varepsilon$,
for example, 
\[ \varepsilon=
\min_{b\in B} \left\{ \min_{x\in \R^{n}} \left\{ 
\max_{b'\in B\setminus \{b\}} \left\{ 
\Theta_{b'}(x)
\right\} -\Theta_{b}(x)
\right\}\right\}. \]
The convexity of $\operatorname{Im}\varphi$ tells us that 
$\operatorname{Im}\varphi$ includes an $|B|$-simplex.
Therefore the image $\operatorname{Im}\varphi$ is an $|B|$-dimensional convex polyhedron.
\end{proof}

\begin{rmk}
The convexity of $\varphi (\Gamma (X,L))$ is clear by Proposition \ref{prop} and the fact that
the set (\ref{Legendre}) is  convex.
\end{rmk}
\begin{rmk}
Mikhalkin-Zharkov \cite[Remark 5.5]{MZ} 
states (without proof) that the dimension of $\Gamma (X,L)$
is given by the degree of the polarization. 
A novel observation in this paper is that $\Gamma (X,L)$ is a convex polyhedron.
\end{rmk}

\subsection{The space of theta function: positive semidefinite case}
Next, we consider a line bundle $L(Q,\alpha )$ on the tropical torus $X=\R^{n}/\Lambda$ 
with $Q$ positive semidefinite.
Then $\operatorname{Ker}(q_{\R})$ is not necessarily $0$ and
it is both $\Lambda$-rational and $\Z$-rational
because of the symmetry of $Q\in \Lambda^{*}\otimes (\Z^{n})^{*}$.
Here a subspace of $\R^{n}$ is said to be
$\Lambda$-rational (resp. $\Z$-rational) if
 it is generated by elements of $\Lambda$ (resp. $\Z^{n}$).

First, we suppose $\alpha \in \operatorname{Im}(q_{\R})+(\Z^{n})^{*}$.
Then we can describe $\alpha =q_{\R}(r)+\gamma $ 
for some $r\in \R^{n}$ and $\gamma\in (\Z^{n})^{*}$.
Since $L(Q,\alpha )\cong L(Q, q_{\R}(r))$,
we can assume that $\alpha \in\operatorname{Im}(q_{\R})$.

Recall the quasi-periodicity of theta functions for $L(Q,\alpha )$,
\[ \Theta (x+\lambda )=\Theta (x)+Q_{\R}(\lambda ,x)
+\frac{1}{2}Q_{\R}(\lambda ,\lambda ) +\alpha  (\lambda ). \] 
By the quasi-periodicity,
theta functions are $\operatorname{Ker}(q)$-periodic since $\alpha \in\operatorname{Im}(q_{\R})$.
Theta functions are convex and periodic in the direction $\operatorname{Ker}(q_{\R})$,
thus they are constant along $\operatorname{Ker}(q_{\R})$, that is,
\[ \Theta (x+\lambda '' )=\Theta (x)
\ \text{for all }\lambda ''\in \operatorname{Ker}(q_{\R}),\ x\in \R^{n}. \]
Thus theta functions  descends to functions on $\R^{n}/\operatorname{Ker}(q_{\R})$. 
The quotient space $\R^{n}/\operatorname{Ker}(q_{\R})$ contains the two natural lattices,
$\Z^{n}/(\operatorname{Ker}(q_{\R})\cap \Z^{n})$ induced by $\Z^{n}$ and
$\Lambda/\operatorname{Ker}(q)$ induced by $\Lambda$.

We denote a descendant theta function by $\overline{\Theta }$.
This  
satisfies the following quasi-periodicity
for all $x\in \R^{n}/\operatorname{Ker}(q_{\R})$ and $\lambda\in \Lambda /\operatorname{Ker}(q)$;
\[ \overline{\Theta} (x+\lambda )=\overline{\Theta} (x)+\overline{Q_{\R}}(\lambda ,x)+
\frac{1}{2}\overline{Q_{\R}}(\lambda ,\lambda)+\overline{\alpha } (\lambda ). \]
Here $\overline{Q_{\R}}\in (\R^{n}/\operatorname{Ker}(q_{\R}))^{*}\otimes
(\R^{n}/\operatorname{Ker}(q_{\R}))^{*}$
is the symmetric form induced by $Q$ and
$\overline{\alpha }$ is the linear function on $\R^{n}/\operatorname{Ker}(q_{\R})$
induced by $\alpha $.

By the above argument, the tropical module $\Gamma (X, L(Q,\alpha ))$ is identified with
\begin{eqnarray*}
\left\{ \overline{\Theta} \colon \R^{n}/\operatorname{Ker}(q_{\R})\to \R \ \middle|\ 
\begin{aligned}
&\overline{\Theta} (x+\lambda )=\overline{\Theta} (x)+\overline{Q_{\R}}(\lambda ,x)+
\frac{1}{2}\overline{Q_{\R}}(\lambda, \lambda)+ \overline{\alpha } (\lambda )\\
~&\text{for all }\lambda \in \Lambda /\operatorname{Ker}(q) ,
\ x\in \R^{n}/\operatorname{Ker}(q_{\R}).\\
~&\hspace{0mm} \text{$\overline{\Theta} $ is a
$ \left( \Z^{n}/(\operatorname{Ker}(q_{\R})\cap \Z^{n}) \right) ^{*}$-regular function.}
\end{aligned} \right\} \cup \{ -\infty \}.
\end{eqnarray*}
Here a $\left( \Z^{n}/(\operatorname{Ker}(q_{\R})\cap \Z^{n}) \right) ^{*}$-regular function
means a convex piecewise-linear function whose slopes lie in 
$\left( \Z^{n}/(\operatorname{Ker}(q_{\R})\cap \Z^{n}) \right) ^{*}$.

Under the identification
\begin{align*}
 \left( \R^{n}/\operatorname{Ker}(q_{\R})\right) ^{*}
&=\left\{ a\in (\R^{n})^{*}
\ \middle|\ 
a(\operatorname{Ker}(q_{\R}))=0 \right\} \\
 \left( \Z^{n}/(\operatorname{Ker}(q_{\R})\cap \Z^{n}) \right) ^{*}
&=\left\{ a\in (\Z^{n})^{*}
\ \middle|\ 
a(\operatorname{Ker}(q_{\R})\cap \Z^{n})=0 \right\}, 
\end{align*}
we have
\begin{align*}
 \left( \R^{n}/\operatorname{Ker}(q_{\R})\right) ^{*}
&=\operatorname{Im}(q_{\R}) \\
 \left( \Z^{n}/(\operatorname{Ker}(q_{\R})\cap \Z^{n}) \right) ^{*}
&=\operatorname{Im}(q_{\R})\cap (\Z^{n})^{*}\supset  \operatorname{Im}(q)
\end{align*}
by the symmetry of $Q$.


In this situation, we can apply the result of the positive definite case.  
Let 
$\overline{q}\colon \Lambda /\operatorname{Ker}(q) \to
\operatorname{Im}(q_{\R})\cap (\Z^{n})^{*}$
be the linear map given by $q$.
Then we can see that $\Gamma (X,L)$ is generated by 
$l:=\left| \operatorname{Cok}(\overline{q})\right|
=\left| \text{the torsion part of }\operatorname{Cok}(q)\right|$
elements  
and that $\Gamma (X,L)$ naturally embeds into $\T^{l}$.
Moreover, 
$\Gamma (X,L)$ has the structure of a pure $l$ dimensional polyhedron in $\T^{l}$,
and its projectivization is a compact polyhedron.

Next, we suppose that $\alpha \notin \operatorname{Im}(q_{\R})+(\Z^{n})^{*}$.
If $\alpha $ is integral on $\operatorname{Ker}(q_{\R})\cap \Z^{n}$,
that is, $\alpha (\operatorname{Ker}(q_{\R})\cap \Z^{n})\subset (\Z^{n})^{*}$,
we can assume that $\alpha  (\operatorname{Ker}(q_{\R})\cap \Z^{n})=0$
by translating $\alpha $ by an element of $(\Z^{n})^{*}$.
Then we can consider $\alpha $ as a linear map from $\R^{n}/\operatorname{Ker}(q_{\R})$ to $\R$,
that is, $\alpha  \in \operatorname{Im}(q_{\R})$. This is a contradiction.
Therefore,  $\alpha $ is not integral even on $\operatorname{Ker}(q_{\R})\cap \Z^{n}$.

By the quasi-periodicity, theta functions for $L(Q,\alpha )$ must satisfy
\[ \Theta (x+\lambda '' )=\Theta (x)+\alpha  (\lambda '')
\ \text{for all }\lambda ''\in \operatorname{Ker}(q_{\R})\cap \Lambda ,\ x\in \R^{n}. \]
Thus theta functions are affine linear on $\operatorname{Ker}(q_{\R})$ with slope $\alpha $.
Since $\alpha $ is not integral on $\operatorname{Ker}(q_{\R})\cap\Z^{n}$,
theta functions cannot be regular.
Thus $\Gamma (X,L)=\{ -\infty\}$.

Consequently we get the following theorem.
\begin{thm}\label{THM}
Let $X$ be a tropical torus and 
$L=L(Q,\alpha )$ be a tropical line bundle on $X$ 
such that $Q=c_{1}(L)$ is positive semidefinite.
\begin{itemize}
\item[(1)] If $\alpha $ lies in $\operatorname{Im}(q_{\R})+(\Z^{n})^{*}$,
then $\Gamma (X,L)$ is generated by 
$l=\left| \text{the torsion part of }\operatorname{Cok}(q)\right|$ elements 
as the $\T$-module and $\Gamma (X,L)$ embeds into $\T^{l}$ as in Definition \ref{inj}.
Moreover, 
$\Gamma (X,L)$ is identified with an $l$-dimensional convex polyhedron in $\T^{l}$
and its projectivization is identified with a compact polyhedron in $\R^{n}$.
\item[(2)] If $\alpha $ is not in $\operatorname{Im}(q_{\R})+(\Z^{n})^{*}$,
then $\Gamma (X,L)=\{ -\infty \} .$
\end{itemize}
\end{thm}

\subsection{The rank of divisor}
In this section, we show that $h^{0}(X,D)$ introduced in Cartwright \cite{DC2}
 is equal to the topological dimension of $\Gamma (X,\mathcal{O}(D))$.

\begin{dfn}
Let $X$ be a tropical manifold. 
A divisor $D$ on $X$ is {\it effective} if the divisor $D$ is given by a global regular section of
$\mathcal{O}(D)$.
For a divisor $D$ on $X$,
we denote the set of effective divisor linearly equivalent to $D$ by $|D|$.
\end{dfn}

\begin{dfn}[Cartwright \cite{DC2}, Definition 3.1]
Let $X$ be a tropical manifold and $D$ be a divisor on $X$.
We define $h^{0}(X,D)$ as
\[ h^{0}(X,D):=\min\left\{ k\in \Z_{\geq 0} \middle| 
\begin{aligned}
&\text{there exist $k$ points } p_{1},\ldots ,p_{k}, \text{ such that } \\
~&\hspace{0mm}\text{there is no divisor }
E\in |D| \text{ which passes all }p_{i}
\end{aligned}
\right\} .\]
If there is no such $k$, then we define $h^{0}(X,D)$ as $\infty$.
\end{dfn}

\begin{rmk}
If $|D|=\emptyset$, then $h^{0}(X,D)=0$ holds trivially.
\end{rmk}

\begin{rmk}
For a tropical curve $C$, we can check that $h^{0}(C,D)=r(D)+1$.
\end{rmk}

Let $X=\R^{n}/\Lambda $ be a topical torus and let $D$ be a divisor on $X$.
By Proposition \ref{LQA}, $\mathcal{O}(D)$ is isomorphic to $L(Q,\alpha )$
for $Q=c_{1}(\mathcal{O}(D))$ and some $\alpha\in (\R^{n})^{*}$.
Let $q\colon \Lambda \to (\Z^{n})^{*}$ be the linear map given by $Q$ and
let $B=\{ b_{1},\ldots , b_{l} \}\subset (\Z^{n})^{*}$ be a complete set of representatives of
the torsion part of $(\Z^{n})^{*}/q(\Lambda )$.
We take the basis $\{ \Theta_{b} \}_{b\in B}$ of $\Gamma (X,\mathcal{O}(D))$
defined in Section 3.3. 

\begin{thm}\label{rD}
The value $h^{0}(X,D)$ coincides with the topological dimension of 
$\Gamma (X,\mathcal{O}(D))$ as a convex polyhedron.
\end{thm}
\begin{proof}
When the topological dimension of $\Gamma (X,\mathcal{O}(D))$ as
a convex polyhedron is $0$ or $1$, 
it is clear that $h^{0}(X,D)$ coincides with this dimension.

We assume that $\Gamma (X,\mathcal{O}(D))\neq -\infty$ and $l\geq 2$. 
In order to see that $h^{0}(X,D)\geq l$,
we take $l-1$-points $p_{1},\ldots ,p_{l-1}\in X$ and choose $q_{1},\ldots ,q_{l-1}\in \R^{n}$ such that
$p_{i}=q_{i}+\Lambda$.

The tropical determinant $\operatorname{det}_{\text{trop}}(t_{i,j})$
 of a square matrix $(t_{ij})_{1\leq i,j \leq l}$ whose entries are in $\T$ is defined as 
$\max_{\sigma\in \mathfrak{S}_{l}}\{ \sum_{i=1}^{n}a_{\sigma (i)i} \}$,
where $\mathfrak{S}_{l}$ is the symmetric group of $\{ 1,\ldots ,l \}$.

Now we define a theta function for $\mathcal{O}(D)$ by the Vandermonde determinant
\begin{align*}
V_{D}(x; q_{1},\ldots ,q_{l-1})&:=
\operatorname{det}_{\text{trop}}
\left( \begin{array}{cccc} 
\Theta _{b_{1}}(x)&\Theta _{b_{2}}(x)&\cdots &\Theta _{b_{l}}(x)\\ 
\Theta _{b_{1}}(q_{1})&\Theta _{b_{2}}(q_{1})&\cdots &\Theta _{b_{l}}(q_{1})\\
\vdots&\vdots&\ddots&\vdots\\
\Theta _{b_{1}}(q_{l-1})&\Theta _{b_{2}}(q_{l-1})&\cdots &\Theta _{b_{l}}(q_{l-1})
\end{array}\right)\\
&=\max_{\sigma\in \mathfrak{S}_{l}}\left\{ \Theta_{b_{\sigma (1)}}(x)+
\sum_{i=2}^{l}\Theta_{b_{\sigma (i)}}(q_{i-1}) \right\}.
\end{align*}
We can see that the divisor on $X$ given by $V_{D}(x; q_{1},\ldots ,q_{l-1})$
is linearly equivalent to $D$ and passes the $l-1$ points $p_{1},\ldots, p_{l-1}$.
Thus we get $h^{0}(X,D)\geq l$.

Next, we suppose $h^{0}(X,D)\geq l+1$.
Then for any $l$ points $p_{1},\ldots ,p_{l}\in X$,
we can take a theta function $\Theta := \max_{b\in B}\{  \Theta_{b}(x)+t_{b}\}$
which is non-linear at a representative $q_{k}\in \R^{n}$ of $p_{k}$ for all $k=1,\ldots ,l$.
We set $H:= \bigcup_{b\in B}\operatorname{supp}(\operatorname{div} \Theta_{b}) \subset \R^{n}$.
We can assume that $q_{1},\ldots ,q_{l}$ are not in $H$.

Now we define a multigraph whose vertex set is $B$.
For each $k=1,\ldots ,l$, we can choose two distinct elements $b_{i_{k}}, b_{j_{k}}\in B$
such that $\Theta_{b_{i_{k}}}(q_{k})+t_{b_{i_{k}}}=\Theta_{b_{j_{k}}}(q_{k})+t_{b_{j_{k}}}$
since $\Theta$ is singular at $q_{k}$.
We connect $b_{i_{k}}$ and $b_{j_{k}}$ by an edge for each $k=1,\ldots ,l$.
Then we have a multigraph $\Gamma$ with $l$ vertices and $l$ edges.
Such a $\Gamma$ necessarily has a cycle.

Let $b_{1},\ldots ,b_{k}$ make a cycle. 
We can assume that $b_{i}\neq b_{j}$ for $i\neq j$.
Then we get
\begin{align*}
\Theta_{b_{1}}(q_{i_{1}})+t_{b_{1}}&=\Theta_{b_{2}}(q_{i_{1}})+t_{b_{2}}\\
\Theta_{b_{2}}(q_{i_{2}})+t_{b_{2}}&=\Theta_{b_{3}}(q_{i_{2}})+t_{b_{3}}\\
~&\vdots \\
\Theta_{b_{k-1}}(q_{i_{k-1}})+t_{b_{k-1}}&=\Theta_{b_{k}}(q_{i_{k-1}})+t_{b_{k}}\\
\Theta_{b_{k}}(q_{i_{k}})+t_{b_{k}}&=\Theta_{b_{1}}(q_{i_{k}})+t_{b_{1}}
\end{align*}
for some $q_{i_{1}},\ldots q_{i_{k}}$.
By combining these equations, we get
\[
\left( \Theta_{b_{1}}(q_{i_{1}})-\Theta_{b_{2}}(q_{i_{1}})\right)
+\cdots +\left( \Theta_{b_{k-1}}(q_{i_{k-1}})-\Theta_{b_{k}}(q_{i_{k-1}})\right)
+\left( \Theta_{b_{k}}(q_{i_{k}})-\Theta_{b_{1}}(q_{i_{k}})\right)=0.
\]
More generally, we define the functions $l_{i_{1},\ldots ,i_{k}}$ for $k>1$ and
 a subset $\{ i_{1},\ldots ,i_{k}\}\subset \{ 1,\ldots ,l \}$ as follows:
\[ l_{i_{1},\ldots ,i_{k}}(x_{1},\ldots x_{k})
:=\sum_{j=1}^{k}\left(
\Theta_{b_{i_{j}}}(x_{j})-\Theta_{b_{i_{j+1}}}(x_{j	})
\right),
 \]
where $i_{k+1}:=i_{1}$.
These functions are rational functions on $(\R^{n})^{k}$.
To summarize the above argument, we obtain the following lemma.

\begin{lem}\label{l0}
Suppose $h^{0}(X,D)\geq l+1$. 
Then for any $l$ points $q_{1}\ldots ,q_{l}\in \R^{n}\setminus H$,
there exist an ordered sequence $\{ i_{1},\ldots ,i_{k}\}$ of elements in $\{ 1,\ldots ,l \}$
 and $\{ q_{j_{1}}, \ldots, q_{j_{k}}\}\subset \{ q_{1}\ldots ,q_{l} \}$
such that
$ l_{i_{1},\ldots ,i_{k}}(q_{j_{1}},\ldots q_{j_{k}})=0 $.
\end{lem}

We define subsets $H_{1}$, $H_{k}\subset (\R^{n})^{l}$ , $k=2,\ldots ,l$ as below.
\[ 
H_{1}:=\left\{ ( x_{1},\ldots ,x_{l} )\in (\R^{n})^{l} \middle| x_{i} \in H \text{ for some }i \right\}.
\]
\[
H_{k}:=\left\{ ( x_{1},\ldots ,x_{l} )\in (\R^{n})^{l} \middle| 
\begin{aligned}
\text{There exist an ordered sequence }\{ i_{1},\ldots ,i_{k}\}\text{ of elements in } \{ 1,\ldots ,l \}\\
~&\hspace{-87mm}\text{and a subset }\{ x_{j_{1}}, \ldots ,x_{j_{k}} \}\subset \{ x_{1},\ldots ,x_{l} \},\\ 
~&\hspace{-87mm} \text{such that } l_{i_{1},\ldots ,i_{k}}(x_{j_{1}},\ldots x_{j_{k}})=0
\end{aligned}\ \ \ \ \ \ \ \ \ \ \ \ \ \ \ \ 
\right\}.
\]

It is clear that the dimension of $H_{1}$ is $ln-1$.
The dimension of $H_{k}$,for $k=2,\ldots ,l$ is also $ln-1$
since functions $l_{i_{1},\ldots ,i_{k}}$ are rational functions on $(\R^{n})^{l}$ with non-zero slopes.
Thus $\R^{ln}\setminus \cup_{k=1}^{l}H_{k}$ is not empty and 
we can take $l$ points $q_{1},\ldots ,q_{l}$ with $(q_{1},\ldots ,q_{l})\in (\R^{n})^{l}\setminus \cup_{k=1}^{l}H_{k}$.
By definition, for any subset $\{ q_{j_{1}},\ldots q_{j_{k}} \}$ of these $l$ points 
and any ordered sequence $\{ i_{1},\ldots ,i_{k}\}$ of $\{ 1,\ldots ,l \}$,
$ l_{i_{1},\ldots ,i_{k}}(q_{j_{1}},\ldots q_{j_{k}})\neq 0$.
This contradicts Lemma \ref{l0}.
Thus we get $h^{0}(X,D)< l+1$. 
\end{proof}



\section{Riemann-Roch inequality for tropical abelian varieties}\label{intersection number}

\subsection{Intersection number of divisors on a tropical torus}
In this section, we  briefly define
the intersection number of $n$ divisors on an $n$-dimensional tropical torus.
For the details, we refer the reader to 
Allermann-Rau \cite{AR} and  Mikhalkin \cite{M}.

Let $C_{1},C_{2},\ldots , C_{n}$ be divisors in $\R^{n}$.
We take rational polyhedral subdivisions of $C_{1}, \ldots , C_{n}$ such that
the intersection $\bigcap_{i=1}^{n}C_{i}$ is a rational polyhedral subcomplex of each $C_{i}$.
We define $S$ to be the set of vertices of $\bigcap_{i=1}^{n}C_{i}$.

At first, we define the intersection multiplicity at a point $p\in S$.
For each $i=1, \ldots ,n$, we can take an open neighborhood $U_{i}$ of $p$ in $C_{i}$
which is contained in the open star of $p$, so that
$p$ is the only vertex of $\bigcap_{i=1}^{n}U_{i}$.
Let $U_{i}(v_{i})$ be the translation of $U_{i}$ by a vector $v_{i}\in \R^{n}$.
For generic sufficiently small $v_{i}$,
$\bigcap_{i=1}^{n}U_{i}(v_{i})=\emptyset$ or
$U_{1}(v_{1}),\ldots ,U_{n}(v_{n})$ intersect transversally and 
none of intersection points lie in $(n-2)$-dimensional cells of $U_{i}(v_{i})$.

If $\bigcap_{i=1}^{n}U_{i}(v_{i})=\emptyset$, we define the intersection multiplicity at $p$ to be $0$.
If $\bigcap_{i=1}^{n}U_{i}(v_{i})\neq\emptyset$,
for every intersection point $q\in \bigcap_{i=1}^{n}U_{i}(v_{i})$,
there exist (n-1)-dimensional cells $E_{i}$ of $U_{i}(v_{i})$
such that $\{q\}=\bigcap_{i=1}^{n} E_{i}$.
Let $w_{i}$ be the weight of $E_{i}$ and $v_{i}$ be a primitive vector orthogonal to $E_{i}$.
We define the intersection multiplicity at $q$ as
$\operatorname{mult}(q):=w_{1}w_{2}\cdots w_{n}\left| \operatorname{det}(v_{1},v_{2},\ldots ,v_{n})\right|$
and we define the intersection multiplicity at a point $p\in S$ to be 
$\operatorname{mult}(p):=\sum_{q\in \bigcap_{i=1}^{n}U_{i}(v_{i})}\operatorname{mult}(q)$.
We can check that this definition is independent of the choice of $U_{1},\ldots ,U_{n}$ and 
generic sufficiently small vectors $v_{i}$ because of the balancing condition of divisors.
For more details, we refer the reader to \cite[Construction 6.4]{AR} or \cite[Definition 4.4]{M}.

Now we can define
the intersection number $C_{1}.C_{2}.\ldots . C_{n}$ to be
$\sum_{p\in S}\operatorname{mult}(p)$ if the set $S$ is finite.
This definition is independent of the subdivision.
If we take finer subdivisions of $C_{1}, \ldots ,C_{n}$,
a new point of $S$ is on the interior of a cell of $\bigcap_{i=1}^{n}C_{i}$
and its intersection multiplicity is $0$.

Let $X$ be an $n$-dimensional tropical torus
and $D_{1},D_{2},\ldots , D_{n}$ be divisors on $X$.
We take rational polyhedral subdivisions of $D_{1}, \ldots , D_{n}$ as above and
define $S$ to be the set of vertices of $\bigcap_{i=1}^{n}D_{i}$, which is finite set because of 
the compactness of $X$.
Since $X$ locally looks like $\R^{n}$,
we can define the intersection multiplicities of vertices of $S$  in the same way as above.
The intersection number of $D_{1},D_{2},\ldots , D_{n}$ is defined to be 
the sum of the intersection multiplicities of vertices of S.

It is clear that the intersection number is an $n$-linear form on $\operatorname{Div}(X)$.
Moreover, it is invariant under linear equivalence, that is,
\begin{prop}\cite[Lemma 8.3]{AR}
Let $D_{1}, \ldots D_{n-1}$ and $D, D'$ be a divisor on a tropical torus $X/\Lambda$.
If $D$ and $D'$ are linearly equivalent, then
$D_{1}.\ldots D_{n-1}. D=D_{1}.\ldots D_{n-1}. D'$.
\end{prop}
\begin{proof}
We take a rational function $f: \R^{n}/\Lambda \to \R$ such that $D(f)=D-D'$
and define a rational function $f_{t}$ on $\R^{n}/\Lambda$ for $t\in \R$ to be
$\max \{ t, f \}$.
Since the divisor $D(f_{t})$ deforms continuously, the intersection number
$D_{1}.\ldots D_{n-1}.D(f_{t})$ is continuous for $t$, in particular, it is constant for $t$.

Let $t_{-},t_{+}$ be a real number such that $f_{t_{-}}=t_{-}, f_{t_{+}}=f$.
Then
\begin{align*}
D_{1}.\ldots D_{n-1}.D(f)
&=D_{1}.\ldots D_{n-1}.D(f_{t_{+}})\\
&=D_{1}.\ldots D_{n-1}.D(f_{t_{-}})\\
&=D_{1}.\ldots D_{n-1}.0=0.
\end{align*}
\end{proof}

\subsection{Self-intersection number of divisors}

For a symmetric form $Q\in \Lambda^{*}\otimes (\Z^{n})^{*}$,
we define $\det Q$ as the determinant of the matrix representation of $Q$ with respect to
a basis of $(\Z^{n})^{*}$ and a basis of $\Lambda^{*}$
which define the same orientation of $(\R^{n})^{*}$.

\begin{thm}\label{morethm}
Let $X=\R^{n}/\Lambda$ be a tropical abelian variety and $D$ be a divisor on X.
Then $\frac{1}{n!}D^{n}=\det c_{1}(\mathcal{O}(D))$.
In particular, $\frac{1}{n!}D^{n}$ is always an integer.
\end{thm}
\begin{proof}
Firstly, we calculate the self-intersection number of a divisor $D$
for the line bundle $L(E_{n}, 0)$ on the tropical torus $\R^{n}/\Z^{n}$.
We can assume that
the divisor $D$ is a parallel translate of 
$\bigcup_{i=1}^{n}\{ (x_{1},\ldots ,x_{n})\in \R^{n} | x_{i}\in \Z \}$ with weight $1$.
To calculate the self-intersection number $D^{n}$,
we consider the intersection number 
$(D+\varepsilon e_{1}).(D+\varepsilon e_{2}).\ldots .(D+\varepsilon e_{n})$,
where $e_{1},\ldots ,e_{n}\in \Z^{n}$ is the standard basis and 
$\varepsilon$ is a sufficiently small positive number.
Then we can see that an intersection points of $(D+\varepsilon e_{1}), \ldots ,(D+\varepsilon e_{n-1})$, and
$(D+\varepsilon e_{n})$ corresponds to a permutation of numbers $\{ 1,\ldots ,n \}$.
Since the intersection multiplicity of each intersection point is $1$, we get $D^{n}=n!$.
Figure \ref{figthm2} is a $2$-dimensional example of these calculations

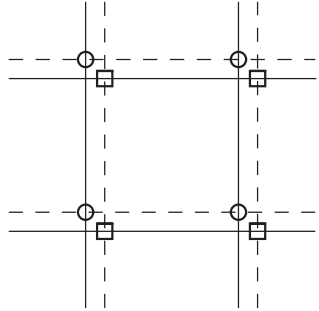
\begin{figure}[h]
\center
\unitlength 0.1in
\begin{picture}( 16.0000, 16.0000)(  0.0000,-18.0000)
%
{\color[named]{Black}{%
\special{pn 8}%
\special{pa 400 200}%
\special{pa 400 1800}%
\special{fp}%
\special{pa 0 1400}%
\special{pa 1600 1400}%
\special{fp}%
\special{pa 1600 600}%
\special{pa 1600 600}%
\special{fp}%
\special{pa 0 600}%
\special{pa 1600 600}%
\special{fp}%
\special{pa 1200 200}%
\special{pa 1200 1800}%
\special{fp}%
}}%
%
{\color[named]{Black}{%
\special{pn 8}%
\special{pa 0 500}%
\special{pa 1600 500}%
\special{da 0.070}%
\special{pa 1600 1300}%
\special{pa 0 1300}%
\special{da 0.070}%
\special{pa 500 200}%
\special{pa 500 1800}%
\special{da 0.070}%
\special{pa 1300 1800}%
\special{pa 1300 200}%
\special{da 0.070}%
}}%
%
{\color[named]{Black}{%
\special{pn 13}%
\special{ar 400 500 40 40  0.0000000  6.2831853}%
}}%
%
{\color[named]{Black}{%
\special{pn 13}%
\special{ar 1200 500 40 40  0.0000000  6.2831853}%
}}%
%
{\color[named]{Black}{%
\special{pn 13}%
\special{ar 1200 1300 40 40  0.0000000  6.2831853}%
}}%
%
{\color[named]{Black}{%
\special{pn 13}%
\special{ar 400 1300 40 40  0.0000000  6.2831853}%
}}%
%
{\color[named]{Black}{%
\special{pn 13}%
\special{pa 460 560}%
\special{pa 540 560}%
\special{pa 540 640}%
\special{pa 460 640}%
\special{pa 460 560}%
\special{pa 540 560}%
\special{fp}%
}}%
%
{\color[named]{Black}{%
\special{pn 13}%
\special{pa 1260 560}%
\special{pa 1340 560}%
\special{pa 1340 640}%
\special{pa 1260 640}%
\special{pa 1260 560}%
\special{pa 1340 560}%
\special{fp}%
}}%
%
{\color[named]{Black}{%
\special{pn 13}%
\special{pa 1260 1360}%
\special{pa 1340 1360}%
\special{pa 1340 1440}%
\special{pa 1260 1440}%
\special{pa 1260 1360}%
\special{pa 1340 1360}%
\special{fp}%
}}%
%
{\color[named]{Black}{%
\special{pn 13}%
\special{pa 460 1360}%
\special{pa 540 1360}%
\special{pa 540 1440}%
\special{pa 460 1440}%
\special{pa 460 1360}%
\special{pa 540 1360}%
\special{fp}%
}}%
\end{picture}%
\caption{Calculating the self-intersection number of
the divisor associated to $L(E_{2},0)$}
\label{figthm2}
\end{figure}

Second, we consider a divisor $D$ for a line bundle
$L(Q, 0)$ where $Q$ is positive definite and principal.
We consider a family of symmetric matrices  $Q_{s}:= ((1-s)E_{n}+sQ^{-1})^{-1} $
and a family of lattices $\Lambda_{s}=q_{s}^{-1}(\Z^{n})^{*}$
for $s\in (-\varepsilon , 1+\varepsilon )$.
Here $q_{s}^{-1}: (\Z^{n})^{*}\to \Lambda_{s}\subset \R^{n}$
is the linear map induced by $Q_{s}^{-1}$ and
$\varepsilon $ is a sufficiently small number such that $Q_{s}$ can be defined.

We define a $(\Z^{n}) ^{*}$-action on $\R^{n}\times (-\varepsilon , 1+\varepsilon )$ by
\[ a\cdot (x,s):= ( Q_{s}^{-1}\cdot a+x,s) , \ a\in (\Z^{n})^{*}, (x,s)\in \R^{n}\times (-\varepsilon , 1+\varepsilon ) . \]
We define $\mathscr{X}$ to be the quotient of $(\R^{n}\times (-\varepsilon , 1+\varepsilon ))$
by the $(\Z^{n})^{*}$-action and $X_{s}:=(\R^{n}\times \{ s\})/(\Z^{n})^{*}=\R^{n}/\Lambda_{s} $.
$\mathscr{X}$ is a tropical manifold and
$X_{s}$ is a divisor of $\mathscr{X}$.

Let $\Theta (x,s)$ be a regular function of $\R^{n}\times (-\varepsilon , 1+\varepsilon )$ defined to be
\begin{align*}
\Theta (x,s)&:=\max_{\lambda\in q_{s}^{-1}((\Z^{n})^{*})}\left\{ 
Q_{s}(\lambda ,x)-\frac{1}{2}Q_{s}(\lambda , \lambda )
\right\} \\
&=\max_{a\in (\Z^{n})^{*}}\left\{
a\cdot x-\frac{1}{2}Q_{s}( q_{s}^{-1}(a), q_{s}^{-1}(a))
\right\} \\
&=\max_{a\in (\Z^{n})^{*}}\left\{
a\cdot x-\frac{1}{2}Q_{s}^{-1}(a,a)
\right\} . \\
\end{align*}
Then $\Theta (x,s)$ satisfies the following quasi-periodicity;
\begin{align*}
\Theta (x+ q_{s}^{-1}(a),s)
&=\max_{a'\in (\Z^{n})^{*}}\left\{
a'\cdot (x+q_{s}^{-1}(a))-\frac{1}{2}Q_{s}^{-1}(a',a')
\right\} \\
&=\max_{a'\in (\Z^{n})^{*}}\left\{
a'\cdot x+Q_{s}^{-1}(a', a)-\frac{1}{2}Q_{s}^{-1}(a',a')
\right\} \\
&=\max_{a'\in (\Z^{n})^{*}}\left\{
a'\cdot x-\frac{1}{2}Q_{s}^{-1}(a'-a,a'-a)+\frac{1}{2}Q_{s}^{-1}(a, a)
\right\} \\
&=\max_{a'\in (\Z^{n})^{*}}\left\{
(a'+a)\cdot x-\frac{1}{2}Q_{s}^{-1}(a',a')+\frac{1}{2}Q_{s}^{-1}(a, a)
\right\} \\
&=\max_{a'\in (\Z^{n})^{*}}\left\{
a'\cdot x-\frac{1}{2}Q_{s}^{-1}(a',a')
\right\} +a\cdot x+\frac{1}{2}Q_{s}^{-1}(a, a)\\
&=\Theta (x,s)+a\cdot x+\frac{1}{2}Q_{t}^{-1}(a,a)
\end{align*}
Thus $\Theta (x,s)$ gives a divisor $[\Theta (x,s) ]$ on $\mathscr{X}$.

Since the intersection number  $[\Theta (x,s)]^{n}.X_{t}$ is constant for $t$, we get
\[ D^{n}=[\Theta (x,s)]^{n}.X_{1}=[\Theta (x,s)]^{n}.X_{0}=n!.\]

Next, we consider the case where $Q$ is positive definite but not necessary principal.
Let $\Xi$ be the section of $\mathcal{O}(D)$ satisfying $\varphi (\Xi ) = 0$ 
and let
$\widetilde{\Lambda}\subset \Lambda\otimes \R$
be the lattice $q_{\R}^{-1}((\Z ^{n})^{*}) $,
which is an overlattice of $\Lambda$.
Then 
$\Xi$ satisfies the quasi-periodicity
\[ \Xi (x+\eta )=
\Xi (x)+Q_{\R}(\eta ,x)+\frac{1}{2}Q_{\R} (\eta ,\eta )+Q_{\R}(\eta ,r) 
,\ \eta\in \widetilde{\Lambda} ,\ x\in \R^{n} \]
and thus
the divisor $\left[ \Xi\right]$ is $\widetilde{\Lambda }$-periodic,
while $\left[\Theta_{b}\right]$ is $\Lambda $-periodic.
Since the self-intersection number of $\left[\Xi\right]$ 
on $\R^{n}/\widetilde{\Lambda}$ is $n!$,
 the self-intersection number of $\left[\Theta_{b}\right]$ on $\R^{n}/\Lambda$ is
$n!\cdot | \widetilde{\Lambda}/\Lambda |=n!\det Q$.

Finally, we show the statement for a divisor $D$ which is not necessary positive definite.
Let $\mathcal{L}$ be the abelian group $\operatorname{Im}c_{1}\times (\R^{n})^{*}$
and $\mathcal{L}_{\R}$ be the linear space $(\operatorname{Im}c_{1}\otimes \R) \times (\R^{n})^{*}$.
An element $(Q,\alpha )$ of $\mathcal{L}$ gives rise to a line bundle $L(Q,\alpha)$.

Since $X$ is a tropical abelian variety, a line bundle $L(Q,\alpha)$ on $X$ has a rational section.
We define an $n$-linear form $\eta$ on $\mathcal{L}^{n}$
to be $\eta ((Q_{1},\alpha_{1}),\ldots ,(Q_{n},\alpha_{n}))=\frac{1}{n!}(D_{1}.\ldots .D_{n})$.
Here $D_{i}$ is a divisor on $X$ given by a rational section of $L(Q_{i},\alpha_{i})$.
This $\eta$ is well-defined because
the intersection number depends only on the linear equivalence class of divisors.
Let $\eta_{\R}\colon \mathcal{L}_{\R}^{n}\to \R$ be the natural extension of $\eta$.
Then $\eta_{\R} ((Q,\alpha), \ldots ,(Q,\alpha))=\det Q$ when $Q$ is positive definite.

That is, $\eta_{\R}((Q,\alpha),\ldots ,(Q,\alpha) ) $ and $\det Q$ coincide on
an intersection of $\mathcal{L}$ and an open cone in $\mathcal{L}_{\R}$.
Since both of them are polynomials of degree $n$,
they must coincide on the whole space $\mathcal{L}_{\R}$.
Therefore $\frac{1}{n!}D^{n}=\det c_{1}(\mathcal{O}(D))\in \Z$ for all divisor $D$ on $X$.

\end{proof}

\begin{cor}\label{cor}
Let $X$ be a tropical abelian surface and $D$ be a divisor on $X$.
Then the Riemann-Roch inequality
\[ h^{0}(X,D)+h^{0}(X,-D)\geq\frac{1}{2}D^{2}. \]
holds.
More precisely,
when we write $Q=c_{1}(\mathcal{O}(D))$ and $\mathcal{O}(D)=L(Q,\alpha)$,
\begin{itemize}
\item[(1)] If $Q$ is positive or negative definite, 
\[ h^{0}(X,D)+h^{0}(X,-D)=\frac{1}{2}D^{2}=\det Q. \]
\item[(2)] If $Q$ is positive or negative semidefinite and $Q$ has $0$ as an eigenvalue, 
\begin{itemize}
\item If $Q\ne 0$ and $\alpha$ lies in $\operatorname{Im}q_{\R}+(\Z^{2})^{*}$,
\[ \left| \text{the torsion part of }\operatorname{Cok}(q)\right|=h^{0}(X,D)+h^{0}(X,-D)
>\frac{1}{2}D^{2}=0, \]
where $q\colon \Lambda \to (\Z^{n})^{*}$ is the linear map given by $Q=c_{1}(\mathcal{O}(D))$
and $q_{\R}\colon \R^{n}\to (\R^{n})^{*}$ is the extension of $q$.
\item If $Q=0$ and $\alpha$ lies in $(\Z^{2})^{*}$,
\[  2=h^{0}(X,D)+h^{0}(X,-D) >\frac{1}{2}D^{2}=0, \]
\item If $\alpha$ does not lie in $\operatorname{Im}q_{\R}+(\Z^{2})^{*}$,
\[ h^{0}(X,D)+h^{0}(X,-D)=\frac{1}{2}D^{2}=0. \]
\end{itemize}
\item[(3)] If $Q$ is indefinite, 
\[ 0=h^{0}(X,D)+h^{0}(X,-D)> \frac{1}{2}D^{2}=\det Q. \]
\end{itemize}
\end{cor}

\begin{proof}
The values of right-hand side was calculated in the above theorem.
On the other hand,
we get the value of left-hand side by combining Theorem \ref{rD} and Theorem \ref{THM}.
\end{proof}

\begin{rmk}
The above result gives us a desirable value of $h^{1}(X,D)$, 
although we don't have any reasonable definition of it.
\end{rmk}

\begin{rmk}
These consequences are compatible with the Riemann-Roch inequality for classical abelian varieties.
See \cite[Corollary 3.5.4, Theorem 3.6.1 and 3.6.3]{CAV}.
\end{rmk}

\section*{Acknowledgements}
I am deeply grateful to my advisor Hiroshi Iritani for his advice.
This paper would not have been possible without his guidance.

Special thanks go to Yuji Odaka for his helpful advice. 
In particular, he suggested to study the Riemann-Roch inequality
for surfaces with trivial canonical class.

I would like to thank Dustin Cartwright for his valuable comments.
He informed me of his definition of $h^{0}(X,D)$ and
suggested to study the relationship between
$h^{0}(X,D)$ and $\dim H^{0}(X,\mathcal{O}(D))$,
and also to study the integrality of $\frac{1}{n!}D^{n}$.
Theorem \ref{rD} and Theorem \ref{morethm} were then obtained.



\end{document}